\documentclass[11pt]{article}
\usepackage{mathrsfs,amsmath,amsfonts,amssymb,bm,bbm,eufrak,dsfont,pifont,amscd,
stmaryrd,euscript,amsthm,appendix,color,epsfig,xr}

\usepackage{graphicx,color,authblk}
\usepackage{tikz}

\usepackage[utf8]{inputenc} 
\usepackage[T1]{fontenc}    
\usepackage{hyperref}       
\usepackage{url}            
\usepackage{booktabs}       
\usepackage{amsfonts}       
\usepackage{nicefrac}       
\usepackage{microtype}      
\usepackage{amsthm}

\usepackage{subfigure} 
\usepackage{comment,color}
\usepackage{bm}
\usepackage{wrapfig}

\newtheorem{theorem}{Theorem}
\newtheorem{lemma}[theorem]{Lemma}
\newtheorem{corollary}[theorem]{Corollary}
\newtheorem{proposition}[theorem]{Proposition}
\theoremstyle{remark}
\newtheorem{remark}{Remark}
\theoremstyle{example} 

\newtheorem{definition}[theorem]{Definition}

\theoremstyle{remark}

\newcommand{\fE}{{\frak E}}

\newcommand{\cd}{\cdot}
\newcommand{\tx}{\tilde{x}}

\renewcommand{\H}{{\mathcal{H}}}
\newcommand{\E}{{\mathbb{E}}}

\newcommand{\tf}{{\tilde{f}}}
\newcommand{\N}{\mathbb{N}}
\newcommand{\R}{\mathbb{R}}

\newcommand{\ve}{\varepsilon}

\newcommand{\bmw}{{\bm w}}
\newcommand{\bmz}{{\bm z}}
\newcommand{\bmalpha}{{\bm \alpha}}
\newcommand{\bmf}{{\bm f}}

\newcommand{\blue}{\textcolor{blue}}

\newcommand{\citep}{\cite}
\newcommand{\citealp}{\cite}

\newtheorem{assumption}{Assumption}

\title{Convergence Analysis of Deterministic Kernel-Based Quadrature Rules in Misspecified Settings}
\vspace{4mm}
\author[1]{Motonobu Kanagawa\thanks{kanagawa@ism.ac.jp/motonobu.kanagawa@gmail.com}}
\author[2]{Bharath K. Sriperumbudur\thanks{bks18@psu.du}}
\author[3]{Kenji Fukumizu\thanks{fukumizu@ism.ac.jp}}
\affil[1]{University of T\"ubingen and Max Planck Institute for Intelligent Systems\\ Max-Planck-Ring 4, 72076 T\"ubingen, Germany}\vspace{2mm}
\affil[2]{Department of Statistics\\ Pennsylvania State University\\
University Park, PA 16802\\ USA}\vspace{2mm}
\affil[3]{The Institute of Statistical Mathematics, 10-3 Midori-cho, Tachikawa, Tokyo 190-8562, Japan}

\begin{document}

\maketitle

\begin{abstract}
This paper presents convergence analysis of kernel-based quadrature rules in misspecified settings, focusing on deterministic quadrature in Sobolev spaces.
In particular, we deal with misspecified settings where a test integrand is less smooth than a Sobolev RKHS based on which a quadrature rule is constructed.
We provide convergence guarantees based on two different assumptions on a quadrature rule: one on quadrature weights, and the other on design points.
More precisely, we show that convergence rates can be derived (i) if the sum of absolute weights remains constant (or does not increase quickly), or (ii) if the minimum distance between design points does not decrease very quickly. 
As a consequence of the latter result, we derive a rate of convergence for Bayesian quadrature in misspecified settings.
We reveal a condition on design points to make Bayesian quadrature robust to misspecification, and show that, under this condition, it may adaptively achieve the optimal rate of convergence in the Sobolev space of a lesser order (i.e., of the unknown smoothness of a test integrand), under a slightly stronger regularity condition on the integrand.
\end{abstract}
\setlength{\parskip}{4pt}

\section{Introduction}
\label{sec:intro}

This paper discusses the problem of numerical integration (or quadrature), which has been a fundamental task in numerical analysis, statistics, computer science including machine learning and other areas. Let $P$ be a (known) Borel probability measure on the Euclidian space $\R^d$ with support contained in an open set $\Omega \subset \R^d$, and $f$ be an integrand on $\Omega$. Suppose that the integral $\int f(x)dP(x)$ has no closed form solution.
We consider quadrature rules that provide an approximation of the integral, in the form of a weighted sum of function values
\begin{equation} \label{eq:intro_numint}
\sum_{i=1}^n w_i f(X_i) \approx \int f(x)dP(x),
\end{equation}
where $X_1,\dots,X_n \in \Omega$ are design points and  $w_1,\dots,w_n \in \R$ are quadrature weights.
Throughout this paper, the integral of $f$ and its quadrature estimate are denoted by $Pf$ and $P_nf$, respectively; namely,
\begin{equation} \label{eq:dist_notation}
Pf := \int f(x)dP(x), \quad P_n f := \sum_{i=1}^n w_i f(X_i). 
\end{equation} 

Examples of such quadrature rules include Monte Carlo methods, which make use of a random sample from a suitable proposal distribution as $X_1,\dots,X_n$, and importance weights as $w_1,\dots,w_n$.  
A limitation of standard Monte Carlo methods is that a huge number of design points (i.e.,~large $n$) may be needed for providing an accurate approximation of the integral; this comes from the fact that the rate of convergence of Monte Carlo methods is typically of the order $\mathbb{E}[|Pf- P_nf|] = O(n^{-1/2})$ as $n \to \infty$, where $\mathbb{E}[\cdot]$ denotes the expectation with respect to the random sample.
The need for large $n$ is problematic, when an evaluation of the function value $f(x)$ is expensive for each input $x$. 
Such situations appear in modern scientific and engineering problems where the mapping $x \mapsto f(x)$ involves complicated computer simulation.
In applications to time-series forecasting, for instance, $x$ may be a parameter of an underlying system, $f(x)$ a certain quantity of interest in future, and $P$ a prior distribution on $x$.
Then the target integral $\int f(x)dP(x)$ is the predictive value of the future quantity. 
The evaluation of $f(x)$ for each $x$ may require numerically solving an initial value problem for the differential equation, which results in time-consuming computation \citep{pmlr-v70-briol17a}.
Similar examples can be seen in applications to statistics and machine learning, as mentioned below.
In these situations, one can only use a limited number of design points, and thus it is desirable to have quadrature rules with a faster convergence rate, in order to obtain a reliable solution \citep{OatGirCho17}.

\subsection{Kernel-based quadrature rules}

How can we obtain a quadrature rule whose convergence rate is faster than $O(n^{-1/2})$?
In practice, one often has prior knowledge or belief on the integrand $f$, such as smoothness, periodicity, sparsity, and so on.  Exploiting such knowledge or assumption in constructing a quadrature rule $\{ (w_i,X_i) \}_{i=1}^n$ may achieve faster rates of convergence, and such methods have been extensively studied in the literature for decades; see e.g.~\cite{DicKuoSlo13} and \cite{BriOatGirOsbSej15} for review.

This paper deals with quadrature rules using reproducing kernel Hilbert spaces (RKHS) explicitly or implicitly to achieve fast convergence rates; we will refer to such methods as {\em kernel-based quadrature rules} or simply {\em kernel quadrature}. 
As discused in Section \ref{sec:rev_meth}, notable examples include Quasi Monte Carlo methods \citep{Hic98,NovWoz10,DicKuoSlo13,DicNuyPil14}, Bayesian quadrature \citep{Oha91,BriOatGirOsbSej15}, and Kernel herding \citep{CheWelSmo10,BacJulObo12,Chen_ICML2018}. 
These methods have been studied extensively in recent years \citep{SomVia06,BriOatGirOsb15,OatGir16,OatGirCho17,Bac17,Xi_ICML2018,Karvonen_NIPS2018} and 
have recently found applications in, for instance, machine learning and statistics \citep{AvrSinYanMah16,LacLinBac15,GerCho15,BriOatGirOsbSej15,KerHen16,Oates_NIPS2017,paul_alternating_2018}. 

In kernel quadrature,  we make use of available knowledge on properties of the integrand $f$ by assuming that $f$ belongs to a certain RKHS $\H_k$ that possesses those properties (where $k$ is the reproducing kernel), 
and then constructing weighted points $\{ (w_i, X_i) \}_{i=1}^n$ such that the {\em worst case error} in the RKHS
\begin{equation} \label{eq:wce}
e_n(P;\H_k) := \sup_{f \in \H_k: \| f \|_{\H_k} \leq 1} \left| Pf - P_n f\right|
\end{equation}
is made small, where $\| \cd \|_{\H_k}$ is the norm of $\H_k$. 
The use of RKHS is beneficial when compared to other function spaces, as it leads to a closed form expression of the worst case error (\ref{eq:wce}) in terms of the kernel, and thus one may explicitly use this expression for designing $\{ (w_i, X_i) \}_{i=1}^n$  (see Section \ref{sec:closed_form_wce}). 

Note that, in a {\it well-specified case}, that is, the integrand $f$ satisfies $f\in\H_k$, the quadrature error is bounded as 
\[
\left| P_n f - P f \right| \leq \| f \|_{\H_k} e_n(P;\H_k).
\]
This guarantees that, if a quadrature rule satisfies $e_n(P; \H_k) = O(n^{-b})$ as $n \to \infty$ for some $b > 0$, then the quadrature error also satisfies $\left| P_n f - P f \right| = O(n^{-b})$.
Take a Sobolev space $H^r(\Omega)$ of order $r>d/2$ on $\Omega$ as the RKHS $\H_k$, for example.  It is known that optimal quadrature rules achieve $e_n(P;\H_k) =  O(n^{-r/d})$ \citep{Nov88}, and thus $\left| P_n f - P f \right| = O(n^{-r/d})$ holds for any $f\in \H_k$.
As we have $r/d > 1/2$, this rate is faster than Monte Carlo integration; this is the desideratum that has been discussed.

\subsection{Misspecified settings}
This paper focuses on situations where the assumption $f \in \H_k$ is violated, that is, {\em misspecified settings}.
As explained above, convergence guarantees for kernel quadrature rules often assume that $f\in\H_k$.
However, in practice one may lack the full knowledge on the properties on the integrand, and therefore misspecification of the RKHS (via the choice of its reproducing kernel $k$) may occur, that is, $f \notin \H_k$.

Such misspecification is likely to happen when the integrand is a {\em black box function}.
An illustrative example can be found in applications to computer graphics such as 
the problem of illumination integration (see e.g.~\citealp{BriOatGirOsbSej15}), where the task is to compute the total amount of light arriving at a camera in a virtual environment.
This problem is solved by quadrature, with integrand $f(x)$ being the intensity of light arriving at the camera from a direction $x$ (angle).
However, the value of $f(x)$ is only given by simulation of the environment for each $x$, so the integrand $f$ is a black box function.
Similar situations can be found in application to statistics and machine learning.
A representative example is the computation of marginal likelihood for a probabilistic model, which is an important but challenging task required for model selection (see e.g.~\citealp{OatPapGir16}).  
In modern scientific applications where complex phenomena are dealt with (e.g.~climate science), we often encounter situations where the evaluation of a likelihood function, which forms the integrand in marginal likelihood computation, involves an expensive simulation model, making the integrand complex and even black box.

If the integrand is a black box function, there is a trade-off between the risk of misspecification and gain in the rate of convergence for kernel-based quadrature rules; for a faster convergence rate, one may want to use a quadrature rule for a narrower $\H_k$ such as of higher order differentiability, while such a choice may cause misspecification of the function class.  Therefore it is of great importance to elucidate their convergence properties in misspecified situations, in order to make use of such quadrature rules in a safe manner.

\subsection{Contributions}
This paper provides convergence rates of kernel-based quadrature rules in misspecified settings, focusing on {\em deterministic} rules (i.e., without randomization).
The focus of misspecification is placed on the order of Sobolev spaces: the unknown order $s$ of the integrand $f$ is overestimated as $r$, that is, $s \leq r$.

Let $\Omega \subset \R^d$ be a bounded domain with a Lipschitz boundary (see Section~\ref{sec:BQ_well} for definition).  For $r>d/2$, consider a positive definite kernel $k_r$ on $\Omega$ that satisfies the following assumption;
\begin{assumption} \label{assumption:kernel}
The kernel $k_r$ on $\Omega$ satisfies $k_r(x,y) := \Phi(x-y)$, where $\Phi: \R^d \to \R$ is a positive definite function such that  
\[
C_1 (1 + \| \xi \|^2)^{-r} \leq \hat{\Phi} (\xi) \leq C_2 (1 + \| \xi \|^2)^{-r}
\]
for some constants $C_1, C_2 > 0$, where $\hat{\Phi}$ is the Fourier transform of $\Phi$. The RKHS $\H_{k_r}(\Omega)$ is the restriction of $\H_{k_r}(\R^d)$ to $\Omega$ (see Section \ref{sec:definition}).
\end{assumption}
The resulting RKHS $\H_{k_r}(\Omega)$ is norm-equivalent to the standard Sobolev space $H^r(\Omega)$. The Mat\'ern and Wendland kernels satisfy Assumption \ref{assumption:kernel} (see Section \ref{sec:definition}).
 
Consider a quadrature rule $\{ (w_i,X_i) \}_{i=1}^n$ with the kernel $k_r$ such that 
\begin{equation} \label{eq:wce_rate_intro190}
e_n(P;\H_{k_r}(\Omega)) = O(n^{-b}) \quad (n \to \infty).
\end{equation}
We do not specify how the weighted points are generated, but assume  (\ref{eq:wce_rate_intro190}) aiming for wide applicability.  Suppose that an integrand $f:\Omega \to \R$ has partial derivatives up to order $s$ and they are bounded and uniformly continuous. 
If $s \leq r$, the integrand may not belong to the assumed RKHS $\H_{k_r}$, in which case a misspecification occurs.

Under this misspecified setting, two types of assumptions on the quadrature rule $\{ (w_i,X_i) \}_{i=1}^n$ will be considered: one on the quadrature weights $w_1,\dots,w_n$ (Section \ref{sec:upper_weights}), and the other on the design points $X_1,\dots,X_n$ (Section \ref{sec:upper_sep}). 
In both cases, a rate of convergence of the form 
\begin{equation} \label{eq:best_rate_205}
| P_n f - P f | = O(n^{-bs/r}), \quad (n\to \infty)
\end{equation}
will be derived under some additional conditions.
The results guarantee the convergence in the misspecified setting, and the rate is determined by the ratio $s/r$ between the true smoothness $s$ and the assumed smoothness $r$.
As discussed in Section \ref{sec:definition}, the optimal rate of deterministic quadrature rules for the Sobolev space $H^r(\Omega)$ is $O(n^{-r/d})$ \cite{Nov88}.  
If a quadrature rule satisfies this optimal rate (i.e., $b=r/d$), then the rate (\ref{eq:best_rate_205}) becomes $O(n^{-s/d})$ for an integrand $f\in H^s(\Omega)$ ($s<r$), which matches the optimal rate for $H^s(\Omega)$.  

The specific results are summarized as follows:\vspace{-2mm}
\begin{itemize}
\item
In Section \ref{sec:upper_weights}, it is assumed that $\sum_{i=1}^n | w_i | = O(n^{c})$ as $n \to \infty$ for some constant $c \geq 0$.
Note that $c = 0$ is taken if the weights satisfy $\max_{i=1,\dots,n} |w_i| = O(n^{-1})$, an example of which is the equal weights $w_1 = \cdots =w_n = 1/n$. Under this assumption and other suitable conditions, Corollary \ref{coro:rate_weight} shows 
\begin{equation*} 
| P_n f - P f | = O( n^{ - bs/r + c (r-s)/r  } ) \quad (n \to \infty).
\end{equation*}
The rate $O(n^{-bs/r})$ in (\ref{eq:best_rate_205}) holds if $c = 0$.
Therefore this result provides convergence guarantees in particular for equal-weight quadrature rules, such as quasi Monte Carlo methods and kernel herding, in the misspecified setting.\vspace{-1mm}

\item
Section \ref{sec:upper_sep} uses an assumption on design points $X^n := \{X_1,\dots,X_n\}$ in terms of {\em separation radius} $q_{X_n}$, which is defined by 
\begin{equation}\label{eq:separation_radius}
q_{X_n} := \frac{1}{2} \min_{i \neq j} \| X_i - X_j \| .
\end{equation}
Corollary \ref{coro:sob_sepa} shows that, if $q_{X^n} = \Theta(n^{-a})$ as $n \to \infty$ for some $a > 0$, under other regularity conditions, 
\begin{equation} 
| P_n f - P f | = O(n^{- \min( b - a(r-s), as)} ) \quad (n \to \infty).
\end{equation}
The best possible rate is $O(n^{-bs/r})$ when $a = b/r$.
This result provides a convergence guarantee for quadrature rules that obtain the weights $w_1,\dots,w_n$ to give $O(n^{-b})$ for the worst case error with $X_1,\dots,X_n$ fixed beforehand. We demonstrate this result by applying it to Bayesian quadrature, as explained below.
Our result may also provide the following guideline for practitioners: in order to make a kernel quadrature rule robust to misspecification, one should specify the design points so that the spacing is not too small. \vspace{-1mm} 

\item
Section \ref{sec:BQ_quasi} discusses a convergence rate for Bayesian quadrature under the misspecified setting, demonstrating the results of Section \ref{sec:upper_sep}.  Given design points $X^n=\{X_1,\dots,X_n\}$, Bayesian quadrature defines weights $w_1,\ldots,w_n$ as the minimizer of the worst case error \eqref{eq:wce}, which can be obtained by solving a linear equation (see Section \ref{sec:rev_meth} for more detail).  For points $X^n=\{X_1,\dots,X_n\}$ in $\Omega$, the {\em fill distance} $h_{X^n,\Omega}$ is defined by  
\begin{equation}\label{eq:fill_dist}
h_{X^n, \Omega} := \sup_{x \in \Omega} \min_{i=1,\dots,n} \| x - X_i \|.
\end{equation}
Assume that there exists a constant $c_q > 0$ independent of $X^n$ such that
\begin{equation} \label{eq:quasi_uni_intro247}
h_{X^n,\Omega} \leq c_q q_{X^n},
\end{equation}
and that $h_{X^n,\Omega} = O(n^{- 1/d})$ as $n \to \infty$.
Then Corollary \ref{coro:BQ_misspecified_rate} shows that with Bayesian quadrature weights based on the kernel $k_r$ we have 
\begin{equation*} 
\left| P_n f - Pf \right| = O(n^{ - s/d }) \quad (n \to \infty). 
\end{equation*}
Note that the rate $O(n^{ - s/d })$ matches the minimax optimal rate for deterministic quadrature rules in the Sobolev space of order $s$  \citep{Nov88}, which implies that Bayesian quadrature can be {\em adaptive} to the unknown smoothness $s$ of the integrand $f$. The adaptivity means that it can achieve the rate $O(n^{-s/d})$ without the knowledge of $s$; it only requires the knowledge of the upper bound of the true smoothness $s \leq r$.\vspace{-1mm}

\item
Section \ref{sec:BQ_well} establishes a rate of convergence for Bayesian quadrature in the {\em well-specified} case, which serves as a basis for the results in the misspecified case (Section \ref{sec:BQ_quasi}).  Corollary \ref{coro:BQ_rate_well} asserts that if the the design points satisfy $h_{X^n, \Omega} = O(n^{-1/d})$ as $n \to \infty$, then 
\[ e_n(P; \H_{k_r}(\Omega)) = O(n^{-r/d}) \quad (n\to \infty).\]
This rate $O(n^{-r/d})$ is minimax optimal for deterministic quadrature rules in Sobolev spaces.  
To the best of our knowledge, this optimality of Bayesian quadrature has not been established before, while recently there has been extensive theoretical analysis on Bayesian quadrature \citep{BriOatGirOsb15,BriOatGirOsbSej15,OatCocBriGir16,Bac17}.

\end{itemize}

This paper is organized as follows.
Section \ref{sec:definition} provides various definitions, notation and preliminaries including reviews on kernel-based quadrature rules.
Section \ref{sec:BQ_well} then establishes a rate of convergence for the worst case error of Bayesian quadrature in a Sobolev space.
Section \ref{sec:sobolev} presents the main contributions on the convergence analysis in misspecified settings, and Section \ref{sec:BQ_quasi} demonstrates these results by applying them to Bayesian quadrature. 
We illustrate the obtained theoretical results with simulation experiments in Section \ref{sec:experiments}.
Finally Section \ref{sec:conclude} concludes the paper with possible future directions.\vspace{-2mm}

\paragraph{Preliminary results.}
This paper expands on preliminary results reported in a conference paper by the authors \citep{KanSriFuk16}.
Specifically, this paper is a complete version of the results presented in Section 5 of \cite{KanSriFuk16}.
The current paper contains significantly new topics mainly in the following points:
(i) We establish the rate of convergence for Bayesian quadrature with deterministic design points, and show that it can achieve minimax optimal rates in Sobolev spaces (Section \ref{sec:BQ_well});
(ii) We apply our general convergence guarantees in misspecified settings to the specific case of Bayesian quadrature, and reveal the conditions required for Bayesian quadrature to be robust to misspecification (Section \ref{sec:BQ_quasi});
To make the contribution (ii) possible, we derive finite sample bounds on quadrature error in misspecified settings (Section \ref{sec:sobolev}).
These results are not included in the conference paper.

We also mention that this paper does not contain the results presented in Section 4 of the conference paper \cite{KanSriFuk16}, which deal with {\em randomized} design points. 
For randomized design points, theoretical analysis can be done based on an approximation theory developed in the statical learning theory literature \citep{CucZho07}. 
On the other hand, the analysis in the deterministic case makes use of the approximation theory developed by \cite{NarWar04}, which is based on Calder\'on's decomposition formula in harmonic analysis \citep{FraJawWei91}.
This paper focuses on the deterministic case, and we will report a complete version of the randomized case in a forthcoming paper.\vspace{-3mm}

\paragraph{Related work.}
The setting of this paper is complementary to that of \cite{OatGir16}, in which the integrand is {\em smoother} than assumed. 
That paper proposes to apply the control functional method by \cite{OatGirCho17} to Quasi Monte Carlo integration, in order to make it adaptable to the (unknown) greater smoothness of the integrand.

Another related line of research is the proposals of quadrature rules that are adaptive to less smooth integrands \citep{Dic07,Dic08,Dic11,FusHanNarWarWri14,GodDic15}.
For instance, \cite{FusHanNarWarWri14} proposed a kernel-based quadrature rule on a finite dimensional {\em sphere}.
Their method is essentially a Bayesian quadrature using a specific kernel designed for spheres. 
They derive convergence rates for this method both in well-specified and misspecified settings, and obtain results similar to ours. The current work differs from \cite{FusHanNarWarWri14} in mainly two aspects: 
(i) quadrature problems considered in standard Euclidean spaces, as opposed to spheres;
(ii) a generic framework is presented, as opposed to the analysis of a specific quadrature rule.
See also a recent work by \cite{Xi_ICML2018}, in which Bayesian quadrature for vector-valued numerical integration is proposed and its adaptability to the less smooth integrands is discussed.

Quasi Monte Carlo rules based on a certain digit interlacing algorithm \citep{Dic07,Dic08,Dic11,GodDic15} are also shown to be adaptive to the (unknown) lower smoothness of an integrand.
These papers assume that an integrand is in an {\em anisotropic} function class in which every function possesses (square-integrable) partial mixed derivatives of order $\alpha \in \N$ in {\em each} variable. 
Examples of such spaces include Korobov spaces, Walsh spaces, and Sobolev spaces of dominating mixed smoothness (see e.g.~\citealp{NovWoz10,DicKuoSlo13}).
In their notation, an integer $d$, which is a parameter called an interlacing factor, can be regarded as an assumed smoothness. 
Then, if an integrand belongs to an anisotropic function class with smoothness $\alpha \in \N$ such that $\alpha \leq d$, the rate of the form $O(n^{-\alpha + \ve})$ (or $O(n^{-\alpha -1/2 + \ve})$ in a randomized setting) is guaranteed for the quadrature error for arbitrary $\ve > 0$.
The present work differs from these works in that (i) isotropic Sobolev spaces are discussed, where the order of differentiability is identical in all directions of variables, and that (ii) theoretical guarantees are provided for generic quadrature rules, as opposed to analysis of specific quadrature methods.

\section{Preliminaries}\label{sec:definition}
In this section, we present the required preliminaries. 
\subsection{Basic definitions and notation}
We will use the following notation throughout the paper. 
The set of positive integers is denoted by $\N$, and $\mathbb{N}_0 := \N \cup \{ 0 \}$. For $\alpha := (\alpha_1,\dots,\alpha_d)^T \in \N_0^d$, we write $| \alpha | := \sum_{i=1}^d \alpha_i$.
The $d$-dimensional Euclidean space is denoted by $\R^d$, and the closed ball of radius $R>0$ centered at $z\in\R^d$ by $B(z,R)$.  
For $a \in \R$, $\lfloor a \rfloor$ is the greatest integer that is less than $a$. 
For a set $\Omega \subset \R^d$, ${\rm diam} (\Omega) := \sup_{x, y \in \Omega} \| x - y\|$ is the diameter of $\Omega$.

Let $p>0$ and $\mu$ be a Borel measure on a Borel set $\Omega$ in $\R^d$.   
The Banach space $L_p(\mu)$ of $p$-integrable functions is defined in the standard way with norm $\| f \|_{L_p(\mu)} = (\int |f(x)|^p d\mu(x))^{1/p}$, and $L_\infty(\Omega)$ is the class of essentially bounded measurable functions on $\Omega$ with norm $\| f \|_{L_\infty(\Omega)} := {\rm ess}\sup_{x \in \Omega} |f(x)|$.
If $\mu$ is the Lebesgue measure on $\Omega \subset \R^d$, we write $L_p(\Omega) := L_p(\mu)$ and further $L_p := L_p(\R^d)$ for $p \in \N \cup \{ \infty \}$.
For $f \in L_1(\R^d)$, its Fourier transform $\hat{f}$ is defined by 
\[ 
\hat{f}(\xi) := \int_{\R^d} f(x) e^{- i \xi^T x} dx, \quad \xi \in \R^d,
\]
where $i := \sqrt{-1}$. 

For $s \in \N$ and an open set $\Omega$ in $\R^d$, $C^s(\Omega)$ denotes the vector space of all functions on $\Omega$ that are continuously differentiable up to order $s$, and  
$C_B^s(\Omega) \subset C^s(\Omega)$ the Banach space of all functions whose partial derivatives up to order $s$ are bounded and uniformly continuous. 
The norm of $C_B^s(\Omega)$ is given by $\| f \|_{C_B^s(\Omega)} := \sum_{\alpha \in \N_0^d: | \alpha | \leq s} \sup_{x \in \Omega} |\partial^\alpha f (x)| $, where $\partial^\alpha$ is the partial derivative with multi-index $\alpha \in \N_0^d$.
The Banach space of the continuous functions that vanish at infinity is denoted by $C_0 := C_0(\R^d)$ with sup norm. Let $C_0^s := C_0^s(\R^d) := C_0(\R^d) \cap C_B^s(\R^d)$ be a Banach space with 
the norm $\| f \|_{C_0^s(\R^d)} := \| f \|_{C_B^s(\R^d)}$.

For function $f$ and a measure $\mu$ on $\R^d$, the support of $f$ and $\mu$ are denoted by ${\rm supp}(f)$ and ${\rm supp} (\mu)$, respectively. The restriction of $f$ to a subset $\Omega\in\R^d$ is denoted by $f|_\Omega$.

Let $F$ and $F^*$ be normed vector spaces with norms $\| \cd \|_F$ and $\| \cd \|_{F^*}$, respectively. 
Then $F$ and $F^*$ are said to be {\em norm-equivalent}, if 
$F= F^*$ as a set, and there exists constants $C_1, C_2 > 0$ such that $C_1 \| f \|_{F^*} \leq \| f \|_{F} \leq C_2 \| f \|_{F^*}$ for all $f \in F$.
For a Hilbert space $\H$ with inner product $\langle\cdot,\cdot\rangle_\H$, the norm of $f\in\H$ is denoted by $\|f\|_\H$.

\subsection{Sobolev spaces and reproducing kernel Hilbert spaces} \label{sec:sobolev_pre}
Here we briefly review key facts regarding Sobolev spaces necessary for stating and proving our contributions; for details we refer to \cite{AdaFou03,Tri06,BreSco08}.
We first introduce reproducing kernel Hilbert spaces. For details, see, e.g., \cite[Section 4]{SteChr2008} and \cite[Section 10]{Wen05}.

Let $\Omega$ be a set.  
A Hilbert space $\H$ of real-valued functions on $\Omega$ is a reproducing kernel Hilbert space (RKHS) if the functional $f\mapsto f(x)$ is continuous for any $x\in\Omega$.  
Let $\langle\cdot,\cdot\rangle_\H$ be the inner product of $\H$.  Then, there is a unique function $k_x\in\H$ such that $f(x)=\langle f,k_x\rangle_\H$.  The kernel defined by $k(x,y):=k_x(y)$ is positive definite, and called reproducing kernel of $\H$.  It is known  
(Moore-Aronszajn theorem \citep{Aronszajn1950}) that for every positive definite kernel $k: \Omega \times \Omega \to \R$ there exists a unique RKHS $\H$ with $k$ as the reproducing kernel. Therefore, the notation $\H_k$ is used to the RKHS associated with $k$.

In the following, we will introduce two definitions of Sobolev spaces, i.e., (\ref{eq:sobolev_def}) and (\ref{eq:Sobolev_fourier}), as both will be used throughout our analysis.

For a measurable set $\Omega \subset \R^d$ and $r \in \mathbb{N}$, a Sobolev space $W_2^r(\Omega)$ of order $r$ on $\Omega$ is defined by 
\begin{equation} \label{eq:sobolev_def}
W_2^r(\Omega)  := \{ f \in L_2(\Omega): D^\alpha f  \in L_2(\Omega)\ {\rm exists\,\,for\,\,all\,}\ \alpha \in \N_0^d\ {\rm with}\ | \alpha | \leq r \},
\end{equation}
where  $D^\alpha f$ denotes the $\alpha$-th weak derivative of $f$. This is a Hilbert space with inner-product 
\[ \left<f, g  \right>_{W_2^r(\Omega)} = \sum_{ |\alpha | \leq r} \left< D^\alpha f, D^\alpha g \right>_{L_2(\Omega)}, \quad f, g \in W_2^r(\Omega).\]
For a positive real $r > 0$, another definition of Sobolev space of order $r$ on $\R^d$ is given by 
\begin{equation} \label{eq:Sobolev_fourier}
H^r(\R^d) := \left\{ f \in L_2(\R^d): \int  |\hat{f}(\xi)|^2  \hat{\Phi}(\xi)^{-1} d\xi < \infty \right\},
\end{equation}
where the function $\hat{\Phi}: \R^d \to \R$ is defined by
\[ 
\hat{\Phi}(\xi) := (1 + \| \xi \|^2 )^{-r}, \quad \xi \in \R^d. 
\]
The inner product of $H^r(\R^d)$ is defined by 
\[
\left< f, g \right>_{H^r(\R^d)} :=  \int   \hat{f}(\xi) \overline{\hat{g}(\xi)}\hat{\Phi}(\xi)^{-1} d\xi , \quad f,g \in H^r(\R^d),
\]
where $\overline{\hat{g}(\xi)}$ denotes the complex conjugate of $\hat{g} (\xi) $.

For a measurable set $\Omega$ in $\R^d$, the (fractional order) Sobolev space $H^r(\Omega)$ is defined by the restriction of $H^r(\R^d)$; namely (see, e.g., \citealp[Eq.~(1.8) and Definition 4.10]{Tri06})
\[
H^r(\Omega) := \left\{ f: \Omega \to \R: f = g|_\Omega,\ \exists\, g \in H^r(\R^d)  \right\}
\]
with its norm defined by
\[
\| f \|_{H^r(\Omega)} := \inf \left\{ \| g \|_{H^r(\R^d)}: g \in H^r(\R^d)\ \text{s.t.}\ f = g|_\Omega  \right\}.
\]
If $r \in \mathbb{N}$ and $\Omega$ is an open set with Lipschitz boundary (see Definition \ref{def:Lipschitz_boundary}), then $H^r(\Omega)$ is norm-equivalent to $W_2^r(\Omega)$ (see, e.g., \citealp[Eqs.~(1.8), (4.20)]{Tri06}).

If $r > d/2$, the Sobolev space $H^r(\R^d)$ is an RKHS \cite[Section 10]{Wen05}. 
In fact, the condition $r > d/2$ guarantees that the function $\hat{\Phi}(\xi) = (1 + \| \xi \|^2)^{-r}$ is integrable, so that $\hat{\Phi}(\xi)$ has a (inverse) Fourier transform
\[
\Phi(x) = \frac{2^{1-r}}{\Gamma(r)} \| x \|^{r-d/2} K_{r-d/2} (\| x \|),
\]
where $\Gamma$ denotes the Gamma function and $K_{r-d/2}$ is the modified Bessel function function of the third kind of order $r-d/2$. The function $\Phi$ is positive definite, and the kernel $\Phi(x-y)$ gives $H^r(\R^d)$ as an RKHS.
This kernel $\Phi(x-y)$ is essentially a Mat\'ern kernel \citep{Mat60,Mat86} with specific parameters.
A {\em Wendland kernel} \citep{Wen95} also defines an RKHS that is norm-equivalent to $H^r(\R^d)$.

\subsection{Kernel-based quadrature rules}
\label{sec:wce}

We briefly review basic facts regarding kernel-based quadrature rules necessary to describe our results. 
For details we refer to \cite{BriOatGirOsbSej15,DicKuoSlo13}.

Let $\Omega \subset \R^d$ be an open set, $k$ be a measurable kernel on $\Omega$, and $\H_k(\Omega)$ be the RKHS of $k$ with inner-product $\left< \cd, \cd\right>_{\H_k(\Omega)}$. 
Suppose $P$ is a Borel probability measure on $\R^d$ with its support contained in $\Omega$, and $\{ (w_i, X_i) \}_{i=1}^n \subset (\R \times \Omega)^n$ is weighted points, which serve for quadrature.
For an integrand $f$, define $Pf := \int f(x)dP(x)$ and $P_nf := \sum_{i=1}^n w_i f(X_i)$ respectively as the integral and a quadrature estimate as in \eqref{eq:dist_notation}.
As mentioned in Section \ref{sec:intro}, a kernel quadrature rule aims at minimizing the worst case error  
\begin{equation} \label{eq:wce_preliminary}
e_n(P;\H_k(\Omega)) := \sup_{f \in \H_k: \| f \|_{\H_k(\Omega)} \leq 1} \left| Pf - P_n f \right|.
\end{equation}
Assume $\int \sqrt{k(x,x)}\,dP(x)<\infty$, and define $m_P, m_{P_n}$\footnote{In the machine learning literature, the function $m_P$ is known as {\em kernel mean embedding}, and the worst case error is called the {\em maximum mean discrepancy}, which have been used in a variety of problems including two-sample testing \citep{SriGreFukSchetal10,GreBorRasSchetal12,MuaFukSriSch17}.}  $\in\H_k(\Omega)$ by 
\begin{equation}
m_P(y) := \int k(y,x)dP(x),  \quad m_{P_n}(y) :=  \sum_{i=1}^n w_i k(y,X_i), \quad y\in \Omega \label{eq:kmean},
\end{equation}
where the integral for $m_P$ is understood as the Bochner integral. 
It is easy to see that, for all $f\in\H$,
\[
P f  = \langle f,m_P\rangle_{\H_k(\Omega)}, \quad P_n f  = \langle f,m_{P_n}\rangle_{\H_k(\Omega)}.
\]
The worst case error (\ref{eq:wce_preliminary}) can then be written as 
\begin{equation} \label{eq:wce_means}
e_n(P;\H_k(\Omega)) = \| m_P - m_{P_n} \|_{\H_k(\Omega)},
\end{equation}
and for any $f\in \H_k(\Omega)$ 
\begin{equation}\label{eq:bounded_well_wce}
| P_n f - P f |  \leq   \| f \|_{\H_k(\Omega)} e_n(P;\H_k(\Omega)). 
\end{equation}
\label{sec:closed_form_wce}
It follows from (\ref{eq:wce_means}) that

\begin{equation} \label{eq:kmean_analytic}
e^2_n(P;\H_k(\Omega))=\int \int k(x,\tx)dP(x)dP(\tx) - 2 \sum_{i=1}^n w_i \int k(x,X_i)dP(x) 
\sum_{i=1}^n \sum_{j=1}^n w_i w_j k(X_i,X_j).
\end{equation}
The integrals in (\ref{eq:kmean_analytic}) are known in closed form for many pairs of $k$ and $P$ (see e.g.\ Table 1 of \cite{BriOatGirOsbSej15}); for instance, it is known if $k$ is a Wendland kernel and $P$ is the uniform distribution on a ball in $\R^d$.
One can then explicitly use the formula (\ref{eq:kmean_analytic}) in order to obtain weighted points $\{ (w_i,X_i) \}$ that minimizes the worst case error (\ref{eq:wce_preliminary}).


\subsection{Examples of kernel-based quadrature rules}  \label{sec:rev_meth}

\paragraph{Bayesian quadrature.}
This is a class of kernel-based quadrature rules that has been studied extensively in literature on statistics and machine learning \citep{diaconis1988bayesian,Oha91,minka2000deriving,GhaZou03,osborne_active_2012,HusDuv12,gunter_sampling_2014,BriOatGirOsb15,BriOatGirOsbSej15,pmlr-v70-briol17a,SarHarSveSan16,Bac17,OatGirCho17}.
In Bayesian quadrature, design points $X_1,\dots,X_n$ may be obtained jointly in a deterministic manner \citep{diaconis1988bayesian,Oha91,minka2000deriving,BriOatGirOsbSej15,SarHarSveSan16}, sequentially (adaptively) \citep{osborne_active_2012,HusDuv12,gunter_sampling_2014,BriOatGirOsb15}, or randomly \citep{GhaZou03,BriOatGirOsbSej15,pmlr-v70-briol17a,Bac17,OatGirCho17}. 
For instance, \cite{BriOatGirOsbSej15} proposes to generate design points randomly as a Markov Chain Monte Carlo sample, or deterministically by a Quasi Monte Carlo rule, specifically as a higher-order digital net \citep{Dic08}.

Given the design points being fixed,  quadrature weights $w_1,\dots,w_n$ are then obtained by the minimization of the worst case error (\ref{eq:kmean_analytic}),
which can be done analytically by solving a linear system of size $n$.
To describe this, let $X_1,\dots,X_n$ be design points such that the kernel matrix $K := (k(X_i,X_j))_{i,j}^n \in \R^{n \times n}$ is invertible. 
The weights are then given by 
\begin{equation} \label{eq:BQ_weight}
{\bm w} := (w_1,\dots,w_n)^T = K^{-1} {\bm z} \in \R^n,
\end{equation}
where  ${\bm z} := ( m_P(X_i) )_{i=1}^n \in \R^n$, with $m_P$ defined in (\ref{eq:kmean}).

This way of constructing the estimate $P_n f$ is called {\em Bayesian quadrature}, since $P_n f$ can be seen as a posterior estimate in a certain Bayesian inference problem with $f$ generated as sample of a Gaussian process (see, e.g., \cite{HusDuv12} and \cite{BriOatGirOsbSej15}).\vspace{-3mm}

\paragraph{Quasi Monte Carlo.}
Quasi Monte Carlo (QMC) methods are equal-weight quadrature rules designed for the uniform distribution on a hyper-cube $[0,1]^d$ \citep{DicKuoSlo13}.
Modern QMC methods make use of RKHSs and the associated kernels to define and calculate the worst case error in order to obtain good design points (e.g.~\citealp{Hic98,SloWoz98,Dic07,DicNuyPil14}).
Therefore, such QMC methods are instances of kernel-based quadrature rules; see \cite{NovWoz10} and \cite{DicKuoSlo13} for a review.\vspace{-3mm}

\paragraph{Kernel herding.}
In the machine learning literature, an equal-weight quadrature rule called {\em kernel herding} \cite{CheWelSmo10} has been studied extensively \citep{HusDuv12,BacJulObo12,LacLinBac15,KanNisGreFuk16}.
It is an algorithm that greedily searches for design points so as to minimize the worst case error in an RKHS.
In contrast to QMC methods, kernel herding may be used with an arbitrarily distribution $P$ on a generic measurable space, given that the integral $\int k(\cd,x)dP(x)$ admits a closed form solution with a reproducing kernel $k$.
It has been shown that a fast rate $O(n^{-1})$ is achievable for the worst case error, when the RKHS is finite dimensional \citep{CheWelSmo10}.
While empirical studies indicate that the fast rate would also hold in the case of an infinite dimensional RKHS, its theoretical proof remains an open problem \citep{BacJulObo12}.

\section{Convergence rates of Bayesian quadrature} \label{sec:BQ_well}

This section discusses the convergence rates of Bayesian quaratuere in well-specified settings.  
It is shown that Bayesian quadrature can achieve the minimax optimal rates for deterministic quadrature rules in Sobolev spaces.   
The result also serves as a preliminary to Section \ref{sec:BQ_quasi}, where misspecified cases are considered.
 

Let $\Omega$ be an open set in $\R^d$ and $X^n := \{ X_1,\dots, X_n\}\subset\Omega$. 
The main notion to express the convergence rate is fill distance $h_{X^n,\Omega}$ (\ref{eq:fill_dist}), 
which plays a central role in the literature on scattered data approximation \citep{Wen05}, and has been used in the theoretical analysis of Bayesian quadrature in \cite{BriOatGirOsbSej15,OatCocBriGir16}.

It is necessary to introduce some conditions on $\Omega$.  The first one is the {\em interior cone condition} \cite[Definition 3.6]{Wen05}, which is a regularity condition on the boundary of $\Omega$.
A {\em cone} $C(x, \xi(x), \theta, R)$ with vertex $x \in \R^d$, direction $\xi(x) \in \R^d$ ($\| \xi(x) \| = 1$), angle $\theta \in (0,2\pi)$ and radius $R > 0$ is defined by
$$C(x, \xi(x), \theta, R) := \{ x + \lambda y:\ y \in \R^d,\ \|y\| = 1,\ \left<y, \xi(x) \right> \geq \cos \theta,\ \lambda \in [0,R] \}.$$
\vspace{-5mm}
\begin{definition}[Interior cone condition] \label{def:interior_cone}
A set $\Omega \subset \R^d$ is said to satisfy an interior cone condition if 
there exist an angle $\theta \in (0,2\pi)$ and a radius $R > 0$ such that every $x \in \Omega$ is associated with a unit vector $\xi(x)$ so that the cone $C(x, \xi(x), \theta, R)$ is contained in $\Omega$.
\end{definition}
The interior cone condition requires that there is no `pinch point' (i.e.~a $\prec$-shape region) on the boundary of $\Omega$; see also \cite{OatCocBriGir16}. 

Next, the notions of special Lipschitz domain \cite[p.181]{Ste70} and Lipschitz boundary\footnote{The definition of the Lipschitz boundary in \cite{BreSco08} is identical to the definition of the {\em minimally smooth boundary} in \citep[p.189]{Ste70}. This boundary condition was introduced by Elias M. Stein to prove the so-called {\em Stein's extension theorem} for Sobolev spaces \citep[p.181]{Ste70}.} are defined as follows (see \citealp[p.189]{Ste70}; \citealp[Definition 1.4.4]{BreSco08}). 


\begin{definition}[Special Lipschitz domain] \label{def:special_Lip}
For $d \geq 2$, an open set $\Omega \subset \R^d$ is called a special Lipschitz domain, if there exists a rotation of $\Omega$, denoted by $\tilde{\Omega}$, and a function $\varphi: \R^{d-1} \to \R$ that satisfy the following:
\begin{enumerate}
\item 
$\tilde{\Omega} = \{  (x,y) \in \R^d: y > \varphi(x) \}$;\vspace{1mm}
\item $\varphi$ is a Lipschitz function such that  $|\varphi(x) - \varphi (x') | \leq M \| x - x' \|$ for all $x,x' \in \R^{d-1}$, where $M > 0$.
\end{enumerate}
The smallest constant  $M$ for $\varphi$ is called the Lipschitz bound of $\Omega$.
\end{definition}

\begin{definition}[Lipschitz boundary] \label{def:Lipschitz_boundary}
Let $\Omega \subset \R^d$ be an open set and $\partial \Omega$ be its boundary.
Then $\partial \Omega$ is called a Lipschitz boundary, if there exist constants $\ve > 0$, $N\in\N$, $M$ > 0, and open sets $U_1,U_2,\dots, U_L \subset \R^d$, where $L \in \mathbb{N} \cup \{ \infty \}$, such that the following conditions are satisfied:
\begin{enumerate}
\item For any $x \in \partial \Omega$, there exists an index $i$ such that $B(x,\ve) \subset U_i$, where $B(x,\ve)$ is the ball centered at $x$ and radius $\ve$;\vspace{1mm}
\item  $U_{i_1} \cap \cdots \cap U_{i_{N+1}} = \emptyset$ for any distinct indices $\{ i_1,\dots,i_{N+1} \}$;\vspace{1mm}
\item For each index $i$, there exists a special Lipschitz domain $\Omega_i \subset \R^d$ with Lipschitz bound $b$ such that $U_i \cap \Omega = U_i \cap \Omega_i$ and $b\leq M$.
\end{enumerate}
\end{definition}

Examples of a set $\Omega$ having a Lipschitz boundary include: (i) $\Omega$ is an open bounded set whose boundary $\partial \Omega$ is $C^1$ embedded in $\R^d$; 
(ii) $\Omega$ is an open bounded convex set \citep[p.189]{Ste70}.

\begin{proposition} \label{prop:BQ_fill}
Let $\Omega \subset \R^d$ be a bounded open set such that an interior cone condition is satisfied and the boundary $\partial\Omega$ is Lipschitz, and $P$ be a probability distribution on $\R^d$ with a bounded density function $p$ such that ${\rm supp}(P) \subset \Omega$.
For $r\in\R$ with $\lfloor r \rfloor > d/2$, $k_r$ is a kernel on $\R^d$ that satisfies Assumption \ref{assumption:kernel} and $\H_{k_r}(\Omega)$ is the RKHS of $k_r$ restricted on $\Omega$.
Suppose that $X^n := \{X_1,\dots,X_n\} \subset \Omega$ are finite points such that $G := (k_r(X_i,X_j))_{i,j=1}^n \in \R^{n \times n}$ is invertible, and $w_1,\dots,w_n$ are the quadrature weights given by (\ref{eq:BQ_weight}).
Then there exist constants $C > 0$ and $h_0 > 0$ independent of $X^n$, such that 
\[ 
e_n(P; \H_{k_r}(\Omega))  \leq C h_{X^n,\Omega}^r,
\]
provided that $h_{X^n,\Omega} \leq h_0$, where $e_n(P; \H_{k_r}(\Omega))$ is the worst case error for the quadrature rule $\{ (w_i,X_i) \}_{i=1}^n$.
\end{proposition}
\begin{proof}
The proof idea is borrowed from \cite[Theorem 1]{BriOatGirOsbSej15}. 
Let $f \in \H_{k_r}(\Omega)$ be arbitrary and fixed.
Define a function $f_n \in \H_{k_r}(\Omega)$ by 
\[ f_n:= \sum_{i=1}^n \alpha_i k_r(\cd,X_i) \]
 where $\bmalpha := (\alpha_1,\dots,\alpha_n)^T = G^{-1} \bmf \in \R^n$ and $\bmf := (f(X_1),\dots,f(X_n)) \in \R^n$.
This function is an interpolant of $f$ on $X^n$ such that $f(X_i) = f_n(X_i)$ for all $X_i \in X^n$  

It follows from the norm-equivalence that $f \in H^r(\Omega)$ and 
\begin{equation} \label{eq:norm_equivalence_110}
\| f \|_{H^r(\Omega)} \leq C_1 \| f \|_{\H_{k_r}(\Omega)},
\end{equation}
where $C_1 > 0$ is a constant.

We see that $\sum_{i=1}^n w_i f(X_i) = \int f_n(x) dP(x)$.
In fact, recalling that the weights ${\bm w} := (w_1,\dots,w_n)^T$ are defined as ${\bm w} = G^{-1} {\bm z}$, where ${\bm z} := (z_1,\dots,z_n)^T$ with $z_i := \int k_r(x,X_i)dP(x)$,
it follows that
\begin{equation*}
\sum_{i=1}^n w_i f(X_i) 
= \bmw^T \bmf = \bmz^T G^{-1} \bmf = \bmz^T \bmalpha  
=\sum_{i=1}^n \alpha_i \int k_r(x,X_i)dP(x)= \int f_n(x) dP(x).
\end{equation*}
Using this identity, we have
\begin{eqnarray}
 \left| \int f(x)dP(x) - \sum_{i=1}^n w_i f(X_i) \right|  
&=& \left| \int f(x)dP(x) - \int f_n(x)dP(x) \right| \nonumber \\
&\leq& \| f - f_n \|_{L_1(\Omega)} \| p \|_{L_\infty(\Omega)} \nonumber 
\\
&\leq & C_0 \| f \|_{H^r(\Omega)} h_{X^n,\Omega}^r \| p \|_{L_\infty(\Omega)} \label{eq:ineq_sob_fill} \\
&\leq & C_0 C_1 \| f \|_{\H_{k_r}(\Omega)} h_{X^n,\Omega}^r \| p \|_{L_\infty(\Omega)}, \label{eq:Sob_equivalence_126} 
\end{eqnarray}
where  
(\ref{eq:ineq_sob_fill}) follows from Theorem 11.32 and Corollary 11.33 in \cite{Wen05} (where we set $m := 0$, $p := 2$, $q := 1$, $k := \lfloor r \rfloor $ and $s := r - \lfloor r \rfloor$), 
and (\ref{eq:Sob_equivalence_126}) from (\ref{eq:norm_equivalence_110}).  
Note that constant $C_0$ depends only on $r$, $d$ and the constants in the interior cone condition (which follows from the fact that Theorem 11.32 in \citealp{Wen05} is derived from Proposition 11.30 in \citealp{Wen05}).  Setting 
$C := C_0 C_1 \| p \|_{\infty}$ completes the proof. \qed
\end{proof}

\begin{remark} \rm
\begin{itemize}
\item 
Typically the fill distance $h_{X^n,\Omega}$ decreases to $0$ as the number $n$ of design points increases.
Therefore the upper bound $C h_{X^n \Omega}^r$ provides a faster rate of convergence for $e_n(P; W_2^r(\Omega))$ by a larger value of the degree $r$ of smoothness.\vspace{-1mm}
\item
The condition $h_{X^n,\Omega} \leq h_0$ requires that the design points $X^n = \{ X_1,\dots,X_n \}$ must cover the set $\Omega$ to a certain extent in order to guarantee the error bound to hold.
This requirement arises since we have used a result from the scattered data approximation literature \citep[Corollary 11.33]{Wen05} to derive the inequality (\ref{eq:ineq_sob_fill}) in our proof.
In the literature such a condition is necessary and we refer an interested reader to Section 11 of \cite{Wen05} and references therein.\vspace{-1mm}
\item
The constant $h_0 > 0$ depends only on the constants $\theta$ and $\R$ in the interior cone condition (Definition \ref{def:interior_cone}). 
The explicit form is $h_0 := Q(\lfloor r \rfloor, \theta) R$, where $Q(\lfloor r \rfloor,\theta) := \frac{ \sin \theta \sin \psi }{8 \lfloor r \rfloor^2 (1 + \sin \theta) (1 + \sin \psi) }$ with $\psi := 2 \arcsin \frac{\sin \theta}{4(1+\sin \theta)}$ \cite[p.199]{Wen05}.\vspace{-1mm}

\end{itemize}
\end{remark}

The following is an immediate corollary to Proposition \ref{prop:BQ_fill}.

\begin{corollary} \label{coro:BQ_rate_well}
Assume that $\Omega$, $P$ and $r$ satisfy the conditions in Proposition \ref{prop:BQ_fill}.
Suppose that $X^n := \{ X_1, \dots, X_n \} \subset \Omega$ are finite points such that $G := (k_r(X_i,X_j))_{i,j=1}^n \in \R^{n \times n}$ is invertible and $h_{X^n, \Omega} = O(n^{- \alpha})$ for some $0 < \alpha \leq 1/d$ as $n \to \infty$, 
and $w_1,\dots,w_n$ are the quadrature weights given by (\ref{eq:BQ_weight}) based on $X^n$.
Then we have 
\begin{equation} \label{eq:rate_BQ_rate_well}
e_n(P;\H_{k_r}(\Omega))  = O(n^{- \alpha r}) \quad (n \to \infty),
\end{equation}
where $e_n(P; \H_{k_r}(\Omega))$ is the worst case error of the quadrature rule $\{ (w_i,X_i) \}_{i=1}^n$.
\end{corollary}

\begin{remark} \rm
\begin{itemize}
\item The result (\ref{eq:rate_BQ_rate_well}) implies that the same rate is attainable for the Sobolev space $H^r(\Omega)$ (instead of $H_{k_r}(\Omega))$:
\begin{equation}  \label{eq:BQ_well_rate_204}
e_n(P;H^r(\Omega))  = O(n^{- \alpha r}) \quad (n \to \infty)
\end{equation}
with (the sequence of) the same weighted points $\{ (w_i,X_i) \}_{i=1}^\infty$.
This follows from the norm-equivalence between $\H_{k_r}(\Omega)$ and $H^r(\Omega)$.\vspace{-1mm}
\item
If the fill distance satisfies $h_{X^n,\Omega} = O(n^{-1/d})$ as $n \to \infty$, then $e_n(P; H^r(\Omega)) = O(n^{- r/d})$.
This rate is minimax optimal for the deterministic quadrature rules for the Sobolev space $H^r(\Omega)$ on a hyper-cube \cite[Proposition 1 in Section 1.3.12]{Nov88}. 
Corollary \ref{coro:BQ_rate_well} thus shows that Bayesian quadrature achieves the minimax optimal rate in this setting.\vspace{-1mm}
\item
The decay rate for the fill distance $h_{X^n,\Omega} = O(n^{-1/d})$ holds when, for example, the design points $X^n = \{ X_1,\dots,X_n \}$ are equally-spaced grid points in $\Omega$.
Note that this rate cannot be improved: if the fill distance decreased at a rate faster than $O(n^{-1/d})$, then $e_n(P; H^r(\Omega))$ would decrease more quickly than the minimax optimal rate, which is a contradiction.
\end{itemize}
\end{remark}

\section{Main results} \label{sec:sobolev}

This section presents the main results on misspecified settings.  Two results based on different assumptions are discussed: one on the quadrature weights in Section \ref{sec:upper_weights}, and the other on the design points in Section \ref{sec:upper_sep}.  The approximation theory for Sobolev spaces developed by \cite{NarWar04} is employed in the results. 

\subsection{Convergence rates under an assumption on quadrature weights}
\label{sec:upper_weights}
\begin{theorem} \label{theo:sob_weight}
Let $\Omega \subset \R^d$ be an open set whose boundary is Lipschitz, $P$ be a probability distribution on $\R^d$ with ${\rm supp}(P) \subset \Omega$,  
$r$ be a real number with $r>d/2$, and $s$ be a natural number with $s \leq r$.  Let  
$k_r$ denote a kernel on $\R^d$ satisfying Assumption \ref{assumption:kernel}, and  $\H_{k_r}(\Omega)$ the RKHS of $k_r$ restricted on $\Omega$.
Then, for any $\{ (w_i,X_i) \}_{i=1}^n \in (\R \times \Omega)^n$, $f \in C_B^s (\Omega) \cap H^s (\Omega) \cap L_1(\Omega)$, and $\sigma > 0$, we have
\begin{eqnarray} 
 | P_n f - P f | 
&\leq&  c_1 \left( \sum_{i=1}^n |w_i| + 1 \right)  \sigma^{-s} \| f \|_{C_B^s(\Omega)}\nonumber\\
&&\qquad\qquad +  c_2 (1+\sigma^2)^{\frac{r-s}{2}}  e_n(P;\H_{k_r}(\Omega)) \| f \|_{H^s(\Omega)}, \label{eq:sob_weight_56}
\end{eqnarray}
where $c_1, c_2 > 0$ are constants independent of $\{ (w_i,X_i) \}_{i=1}^n$, $f$ and $\sigma$.
\end{theorem}
\begin{proof}
We first derive some inequalities used for proving the assertion. It follows from norm-equivalence that $f \in W_2^s(\Omega)$, where $W_2^s(\Omega)$ is the Sobolev space defined via weak derivatives.  
Since $\Omega$ has a Lipschitz boundary, Stein's extension theorem \cite[p.181]{Ste70} guarantees that there exists a bounded linear extension operator $\fE: W_2^s(\Omega) \to W_2^s(\R^d)$ such that 
\begin{eqnarray} 
\fE (f) (x) &=& f(x), \quad \forall x \in \Omega, \label{eq:extension_identity_54} \\
\| \fE (f) \|_{W_2^s(\R^d)} &\leq& C_1 \| f \|_{W_2^s(\Omega)}, \label{eq:norm_extension_55}
\end{eqnarray}
where $C_1$ is a constant independent of the choice of $f$.
From the norm-equivalence and (\ref{eq:norm_extension_55}), there is a constant $C_2$ such that 
\begin{equation}
\| \fE f \|_{H^s(\R^d)}
 \leq C_2 \| f \|_{H^s(\Omega)}. \label{eq:sobo_norm_92}
\end{equation}
Since $f \in L_1(\Omega)$, the extension also satisfies $\fE(f) \in L_1(\R^d)$ \cite[p.181]{Ste70}. In addition, by the construction of $\fE$ \cite[Eqs.(24)(31) on p.191]{Ste70}, one can show \citep[Section 3.2.2]{NarWarWen05} that $\fE$ is also a linear bounded operator from $C_B^s(\Omega)$ to $C_0^s(\R^d)$, that is,
\begin{eqnarray} 
\| \fE f \|_{C_0^s(\R^d)} &\leq& C_3 \| f \|_{C_B^s(\Omega)}, \label{eq:norm_extension_60}
\end{eqnarray}
for some constant $C_3>0$.  Below we write $\tf := \fE(f)$ for notational simplicity.

Let $g_{\sigma} \in H^r (\R^d)$ be the approximate function of $\tf$ defined as (\ref{eq:approx_g}) by Calder{\'o}n's formula (\ref{sec:sob_approx}; we set $f := \tf$). 
The property $\tf \in C_0^s(\R^d) \cap H^s(\R^d) \cap L_1(\R^d)$ enables the use of Proposition 3.7 of \cite{NarWar04} (where we set $k := s$ and $\alpha := 0$; see Proposition \ref{prop:NarWar04_Prop3_7} in \ref{sec:results_NarWar04} for a review), which gives in combination with \eqref{eq:norm_extension_60} that
\begin{equation}
\| \tf - g_{\sigma} \|_{L_\infty (\R^d)} 
\leq C \sigma^{-s} \| \tf \|_{C_0^s (\R^d)} 
\leq  C_4 \sigma^{-s} \| f \|_{C_B^s(\Omega)},  \label{eq:sob_sup}
\end{equation} 
for some constant $C_4>0$ which is independent of $f$. 

From $\tf \in C_0^s(\R^d) \cap H^s(\R^d) \cap L_1(\R^d)$, Lemma \ref{eq:gsigma_norm} in \ref{sec:sob_approx} can be applied, by which together with \eqref{eq:sobo_norm_92} we have 
\begin{equation}
\| g_{\sigma} \|_{H^r (\R^d)} 
\leq C_5' (1+\sigma^2)^{\frac{r-s}{2}} \| \tf \|_{H^s (\R^d)}\leq  C_5 (1+\sigma^2)^{\frac{r-s}{2}} \| f \|_{H^s (\Omega)}  \label{eq:sob_g_norm}
\end{equation}
for some constants $C_5$ and $C_5'$, which are independent of $\sigma$ and $\tf$. 

With the decomposition 
\begin{eqnarray*} 
| P_n f - Pf | \leq \underbrace{| P_n f - P_n g_{\sigma} |}_{(A)} +  \underbrace{| P_n g_{\sigma}  -  P g_{\sigma}|}_{(B)} +  \underbrace{| P g_{\sigma}  - P f |}_{(C)},
\end{eqnarray*}
each of the terms $(A)$, $(B)$ and $(C)$ will be bounded in the following. 

First, the term $(A)$ is bounded as
\begin{eqnarray*}
(A) 
&\leq &   \sum_{i=1}^n | w_i |  \left| f(X_i) - g_{\sigma}(X_i) \right| \\
&=& \sum_{i=1}^n |w_i| \left| \tf(X_i) - g_{\sigma}(X_i) \right| \quad (\because \{ X_i \}_{i=1}^n \subset \Omega \ {\rm and}\ (\ref{eq:extension_identity_54}) )\\
&\leq& \left( \sum_{i=1}^n |w_i| \right)  \| \tf - g_{\sigma} \|_{L_\infty (\R^d)} 
\stackrel{\eqref{eq:sob_sup}}{\leq} C_4  \left( \sum_{i=1}^n |w_i| \right)  \sigma^{-s} \| f \|_{C_B^s(\Omega)}.
\end{eqnarray*}

For the term $(B)$, it follows from the norm equivalence and restriction that for some constant $D$
\begin{equation} \label{eq:g_sigma_bound_130}
\| g_{\sigma}|_\Omega \|_{\H_{k_r}(\Omega)} \leq  D \| g_{\sigma} \|_{H^r(\R^d)}.
\end{equation}
This inequality and \eqref{eq:sob_g_norm} give 
\begin{eqnarray*}
(B) &\leq&  \left\|  g_{\sigma} |_{\Omega} \right\|_{ \H_{k_r} (\Omega) }   \left\|  m_{P_n} - m_P \right\|_{ \H_{k_r} (\Omega) }  \\
& \leq & D \| g_{\sigma} \|_{H^r(\R^d)}   e_n(P;\H_{k_r} (\Omega) )  \\ 
&\leq &  D C_5 (1+\sigma^2)^{\frac{r-s}{2}}  e_n(P;\H_{k_r} (\Omega) )  \| f \|_{H^s (\Omega)}.  
\end{eqnarray*}

Finally, the term $(C)$ is bounded as
\begin{equation*}
(C)
\leq \int \left| g_{\sigma}(x) - \tf(x) \right| dP(x) 
\leq \| g_{\sigma} - \tf \|_{L_\infty(\R^d)} 
\stackrel{(\ref{eq:sob_sup})}{\leq} C_4 \sigma^{-s} \| f \|_{C_B^s(\Omega)}.
\end{equation*}

Combining these three bounds, the assertion is obtained.  \qed
\end{proof}
\begin{remark} 
\begin{itemize}
\item The integrand $f$ is assumed to satisfy $f \in  H^s(\Omega) \cap C_B^s(\Omega) \cap L_1(\Omega)$, which is slightly stronger than just assuming $f \in H^s(\Omega)$.\vspace{-1mm}

\item In the upper-bound (\ref{eq:sob_weight_56}), the constant $\sigma > 0$ controls the tradeoff between the two terms: $c_2 (1+\sigma^2)^{\frac{r-s}{2}}  e_n(P;\H_{k_r}(\Omega)) \| f \|_{H^s(\Omega)}$ and $c_1 \left( \sum_{i=1}^n |w_i| + 1 \right) \cdot \sigma^{-s} \| f \|_{C_B^s(\Omega)}$. In the proof, the integrand $f$ is approximated by a band-limited function $g_\sigma \in H^r(\Omega)$, where $\sigma$ is the highest spectrum that $g_\sigma$ possesses.
Thus the tradeoff in the upper-bound corresponds to the tradeoff between the accuracy of approximation of $f$ by $g_\sigma$ and the penalty incurred on the regularity of $g_\sigma$.
%
\end{itemize}
\end{remark}

The following result, which is a corollary of Theorem \ref{theo:sob_weight}, provides a rate of convergence for the quadrature error in a misspecified setting.
It is derived by assuming certain rates for the quantity $\sum_{i=1}^n | w_i |$ and the worst case error $e_n(P;\H_{k_r})$.
\begin{corollary} \label{coro:rate_weight}
Let $\Omega$, $P$, $r$, $s$, $k_r$, and $\H_{k_r}(\Omega)$ be the same as Theorem \ref{theo:sob_weight}. 
Suppose that $\{ (w_i,X_i) \}_{i=1}^n\in (\R \times \Omega)^n$ satisfies $e_n(P;\H_{k_r}(\Omega)) = O(n^{-b})$ and  $\sum_{i=1}^n | w_i | = O(n^c)$ for some $b > 0$ and $c \geq 0$, respectively, as $n \to \infty$.
Then for any $f \in C_B^s (\Omega) \cap H^s (\Omega) \cap L_1(\Omega)$, we have
\begin{equation} \label{eq:rate_sobw}
| P_n f - P f | = O( n^{ - bs/r + c (r-s)/r  } ) \quad (n \to \infty).
\end{equation}
\end{corollary}

\begin{proof}
Let $\sigma_n := n^\theta > 0$, where $\theta > 0$ will be determined later.  
Plugging $e_n(P;\H_{k_r}(\Omega)) = O(n^{-b})$ and $\sum_{i=1}^n | w_i | = O(n^c)$ to \eqref{eq:sob_weight_56} with $\sigma := \sigma_n$ leads 
\begin{equation*} 
| P_n f - P f| 
= O( n^{c - \theta s} ) + O( n^{ \theta(r-s) - b } ).
\end{equation*}
Setting $\theta = (b+c)/r$, which balances the two terms in the right hand side, completes the proof. \qed
\end{proof}
\begin{remark} 
\begin{itemize}
\item 
The exponent of the rate in (\ref{eq:rate_sobw}) consists of two terms: $-bs/r$ and $c(r-s)/r$.
The first term $-bs/r$ corresponds to a degraded rate from the original $O(n^{-b})$ by the factor of smoothness ratio $s/r$, while 
the second term $c(r-s)/r$ makes the rate slower.
The effect of the second term increases as the constant $c$ or the gap $(r-s)$ of misspecification becomes larger.\vspace{-1mm}

\item 
The obtained rate recovers $O(n^{-b})$ for $r=s$ (well-specified case) regardless of the value of $c$. 
\vspace{-1mm}

\item Consider the misspecified case $r>s$.
If $c>0$, the term $c(r-s)/r$ always makes the rate slower. 
It is thus better to have $c=0$, as in this case we have the rate $O(n^{-bs/r})$ in the misspecified setting.
The weights with $\max_{i=1,\dots,n} |w_i| = O(n^{-1})$, such as equal weights $w_i=1/n$, realize $c=0$. 
\vspace{-1mm}

\item 
As mentioned earlier, the minimax optimal rate for the worst case error in the Sobolev space $H^r(\Omega)$ with $\Omega$ being a cube in $\R^d$ and $P$ being the Lebesgue measure on $\Omega$ is $O(n^{-r/d})$ \cite[Proposition 1 in Section 1.3.12]{Nov88}.
If design points satisfy $b = r/d$ and $c = 0$ in this setting, Corollary \ref{coro:rate_weight} provides the rate $O(n^{-s/d})$ for $f \in H^s(\Omega) \cap C_B^s(\Omega) \cap L_1(\Omega)$.
This rate is the same as the minimax optimal rate for $H^s(\Omega)$, and hence implies some adaptivity to the order of differentiability. \vspace{-1mm}

\item 
The assumption $\sum_{i=1}^n |w_i| = O(n^c)$ can be also interpreted from a probabilistic viewpoint.
Assume that the observation involves noise, $Y_i := f(X_i) + \varepsilon_i\ (i=1,\dots,n)$,  where $\ve_i$ is independent noise with $\E[ \ve_i^2] = \sigma_{\rm noise}^2$ ($\sigma_{\rm noise} > 0$ is a constant) for $i=1,\dots,n$, and that $Y_i$ are used for numerical integration.  
The expected squared error is decomposed as 
\begin{eqnarray*}
 \E_{\varepsilon_1,\dots,\varepsilon_n} \left[ \left( \sum_{i=1}^n w_i Y_i - Pf \right)^2  \right] 
&=&   \E_{\varepsilon_1,\dots,\varepsilon_n} \left[ \left( P_n f - Pf + \sum_{i=1}^n w_i \ve_i \right)^2  \right] \nonumber \\
&=&  \left| P_n f - Pf \right|^2 + \sigma_{\rm noise}^2 \sum_{i=1}^n w_i^2. 
\end{eqnarray*}
In the last expression, the first term $\left| P_n f - Pf \right|^2$ is the squared error in the noiseless case, and the second term $\sigma_{\rm noise}^2 \sum_{i=1}^n w_i^2 $ is the error due to noise.
Since $\sum_{i=1}^n w_i^2 \leq ( \sum_{i=1}^n |w_i| )^2 = O(n^{2c})$, the error in the second term may be larger as $c$ increases.
Hence quadrature weights having smaller $c$ are preferable in terms of robustness to the existence of noise; this in turn makes the quadrature rule more robust to the misspecification of the degree of smoothness.  
\end{itemize}
\end{remark}

Theorem~\ref{theo:sob_weight} and Corollary \ref{coro:rate_weight} require a control on the absolute sum of the quadrature weights $\sum_{i=1}^n |w_i|$. 
This is possible with, for instance, equal-weight quadrature rules that seek for good design points.
However, the control of $\sum_{i=1}^n |w_i|$ could be difficult for quadrature rules that obtain the weights by optimization based on pre-fixed design points.
This includes the case of Bayesian quadrature that optimizes the weights without any constraint. 
To deal with such methods, in the next section we will develop theoretical guarantees that do not rely on the assumption on the quadrature weights, but on a certain assumption on the design points.

\subsection{Convergence rates under an assumption on design points}
\label{sec:upper_sep}

This subsection provides convergence guarantees in a misspecified settings under an assumption on the design points.  The assumption is described in terms of separation radius \eqref{eq:separation_radius}, which is (the half of) the minimum distance between distinct design points.  The separation radius of points $X^n := \{ X_1,\dots, X_n \} \subset \R^d$ is denoted by $q_{X^n}$.  
Note that if $X^n\subset \Omega$ for some $\Omega$, then the separation radius lower bounds the fill distance, i.e., $q_{X^n} \leq h_{X^n,\Omega}$.

Henceforth we will consider a bounded domain $\Omega$, and without loss of generality, we assume that it satisfies ${\rm diam}(\Omega) \leq 1$.
\begin{theorem} \label{theo:sob_sepa}
Let $\Omega \subset \R^d$ be an open bounded set with ${\rm diam}(\Omega) \leq 1$ such that the boundary is Lipschitz, $P$ be a probability distribution on $\R^d$ such that ${\rm supp}(P) \subset \Omega$, $r$ be a real number with $r > d/2$, and $s$ be a natural number with $s \leq r$.  Let $k_r$ denote a kernel on $\R^d$ satisfying Assumption \ref{assumption:kernel}, and $\H_{k_r}(\Omega)$ the RKHS of $k_r$ restricted on $\Omega$.
For any $\{ (w_i,X_i) \}_{i=1}^n \in (\R \times \Omega)^n$ and $f \in C_B^s (\Omega) \cap H^s (\Omega)$, we have
\begin{equation}
\left| P_n f - Pf \right| \leq C \max \left(  \| f \|_{C_B^s(\Omega)}, \| f \|_{H^s(\Omega)}  \right)  \left(  q_{X^n}^{- (r-s) } e_n(P; \H_{k_r}(\Omega)) + q_{X^n}^{s} \right), \label{eq:bound_sob_sepa_379}
\end{equation}
where $C > 0$ is a constant depending  neither on $\{ (w_i,X_i) \}_{i=1}^n$ nor on the choice of $f$, and $e_n(P; \H_{k_r}(\Omega))$ is the worst case error in $\H_{k_r}(\Omega)$ for $\{ (w_i, X_i) \}_{i=1}^n$.
\end{theorem}

\begin{proof}
By the same argument as the first part of the proof for Theorem \ref{theo:sob_weight}, there exists an extension of $f$ to $\tf\in W_2^s(\R^d)\cap C_0^s(\R^d)$ such that
\begin{eqnarray}
\tf(x) &=& f(x), \quad \forall x \in \Omega, \label{eq:extension_identity_268} \\
\| \tf \|_{H^s(\R^d)}  &\leq&  C_1 \| f \|_{H^s(\Omega)}, \label{eq:ext_sobnorm_291} \\
\| \tf \|_{C_0^s(\R^d)} &\leq& C_2 \| f \|_{C_B^s(\Omega)}, \label{eq:norm_extension_282}
\end{eqnarray}
for some positive constants $C_i$ ($i=1,2$). 
Note also that $f\in L^1(\Omega)$, since $f\in C^s_B(\Omega)$ and $\Omega$ is bounded.  This implies $\tf \in L_1(\R^d)$ \cite[p.181]{Ste70}.

From the above inequalities, there is a constant $C_3>0$ independent of the choice of $f$ such that 
\begin{equation}
\max \left( \| \tf \|_{C_0^s(\R^d)}, \| \tf \|_{H^s(\R^d)} \right) 
\leq  C_3  \max \left(  \| f \|_{C_B^s(\Omega)}, \| f \|_{H^s(\Omega)}  \right).  \label{eq:max_norm_bound_306}
\end{equation}

For notational simplicity, write 
\begin{equation}
\sigma_n := \frac{C_d}{q_{X^n}}\label{Eq:sign}
\end{equation}
where $C_d := 24(\frac{\sqrt{\pi}}{3} \Gamma(\frac{d+2}{2}))^{\frac{2}{d+1}}$ with $\Gamma$ being the Gamma function.
From Theorems \ref{theo:NarWar04_Th3_5}
and \ref{theo:NarWar04_Th3_10} in \ref{sec:results_NarWar04} (which are restatements of Theorems 3.5 and 3.10 of \citealp{NarWar04}), there exists a function $\tf_{\sigma_n} \in H^r (\R^d)$ such that
\begin{eqnarray}
 \tf_{\sigma_n}(X_i) &=& \tf(X_i),\quad (i=1,\dots,n), \label{eq:Sob_interp} \\
\| \tf - \tf_{\sigma_n} \|_{L_\infty(\R^d)} &\leq& C_{s,d} \sigma_n^{-s} \max( \| \tf \|_{C_0^s(\R^d)}, \| \tf \|_{H^s(\R^d)}), \label{eq:Sob_approx}
\end{eqnarray}
where $C_{s,d}$ is a constant depending only on $s$ and $d$.
Combining (\ref{eq:Sob_approx}) and (\ref{eq:max_norm_bound_306}) obtains
\begin{equation}
\| \tf - \tf_{\sigma_n} \|_{L_\infty(\R^d)} 
\leq C_4 \sigma_n^{-s}   \max \left(  \| f \|_{C_B^s(\Omega)}, \| f \|_{H^s(\Omega)}  \right), \label{eq:bound_unif_330}
\end{equation}
where $C_4 := C_{s,d} C_3$. 

From Assumption \ref{assumption:kernel} and $\tf \in C_B^s(\R^d) \cap H^s(\R^d) \cap L_1(\R^d)$, Lemma \ref{lemma:Sob_norm_interpolant_78} (see \ref{sec:results_NarWar04}) gives 
\begin{equation*}
\| \tf_{\sigma_n} \|_{ \H_{k_r} (\R^d) } \leq C_{s,d,k_r} \sigma_n^{r-s} \max( \| \tf \|_{C_0^s(\R^d)}, \| \tf \|_{H^s(\R^d)}),
\end{equation*}
where $C_{s,d,k_r}$ is a constant only depending on $r$, $s$, $d$, and $k_r$.
It follows from this inequality and (\ref{eq:max_norm_bound_306}) that 
\begin{equation}
\| \tf_{\sigma_n} \|_{ \H_{k_r} (\R^d) }  
\leq  C_5 \sigma_n^{r-s}  \max \left(  \| f \|_{C_B^s(\Omega)}, \| f \|_{H^s(\Omega)}  \right),   \label{eq:bound_RKHS_main344}
\end{equation}
where $C_5 :=  C_{s,d,k_r} C_3$.

We are now ready to prove the assertion. 
In the decomposition 
\begin{eqnarray*} 
| P_n f - P f | = | P_n \tf - P \tf | \leq \underbrace{| P_n \tf - P_n \tf_{\sigma_n} |}_{(A)} +  \underbrace{| P_n \tf_{\sigma_n}  -  P \tf_{\sigma_n}|}_{(B)} +  \underbrace{| P \tf_{\sigma_n}  - P \tf |}_{(C)},
\end{eqnarray*}
the term (A) is zero from \eqref{eq:Sob_interp}. With $\| \tf_{\sigma_n} |_\Omega \|_{\H_{k_r}(\Omega)} \leq \| \tf_{\sigma_n} \|_{\H_{k_r}(\R^d)}$ (\citep{Aronszajn1950}{}{, Section 5}), 
the term $(B)$ can be bounded as 
\begin{eqnarray*}
(B) &=&  \left| \sum_{i=1}^n w_i \tf_{\sigma_n}|_\Omega (X_i) - \int \tf_{\sigma_n}   |_\Omega (x) dP(x) \right|\\
&\leq&  \left| \left< \tf_{\sigma_n} |_{\Omega}, m_{P_n} - m_P \right>_{\H_{k_r} (\Omega)} \right| \quad  (\because \tf_{\sigma_n} |_\Omega \in \H_{k_r} (\Omega))\\
&\leq&  \left\|  \tf_{\sigma_n} |_\Omega \right\|_{\H_{k_r}  (\Omega)}   e_n(P; \H_{k_r}  (\Omega) ) \nonumber
\end{eqnarray*}
\begin{eqnarray*}
&\leq &  \left\|  \tf_{\sigma_n} \right\|_{ \H_{k_r}  (\R^d)} e_n(P; \H_{k_r}  (\Omega) )  \nonumber\\
&\stackrel{\eqref{eq:bound_RKHS_main344}}{\leq}&  C_5  \sigma_n^{r-s} \max \left(  \| f \|_{C_B^s(\Omega)}, \| f \|_{H^s(\Omega)}  \right) e_n(P; \H_{k_r}  (\Omega) ).
\end{eqnarray*}
The term $(C)$ is upper-bounded as
\begin{equation*}
(C) \leq  \| \tf_{\sigma_n} - \tf \|_{L_\infty(\R^d)}
\stackrel{\eqref{eq:Sob_approx}}{\leq} C_4 \sigma_n^{-s}   \max \left(  \| f \|_{C_B^s(\Omega)}, \| f \|_{H^s(\Omega)}  \right).
\end{equation*}
These bounds complete the proof. \qed 
\end{proof}

\begin{remark}
\begin{itemize}
\item
From $q_{X^n} \leq h_{X^n}$, the separation radius $q_{X^n}$ typically converges to zero as $n\to\infty$.  For the upper bound in (\ref{eq:bound_sob_sepa_379}), the factor $q_{X^n}^{-(r-s)}$ in the first term diverges to infinity as $n\to \infty$, while the second term goes to zero.  
Thus $q_{X^n}$ should decay to zero in an appropriate speed depending on the rate of $e_n(P;\H_{k_r}(\Omega))$, in order to make the quadrature error small in the misspecified setting. 
\vspace{-1mm}

\item
Note that as the gap between $r$ and $s$ becomes large, the effect of the separation radius becomes serious; this follows from the expression $q_{X^n}^{-(r-s)}$.

\end{itemize}
\end{remark}

Based on Theorem \ref{theo:sob_sepa}, we establish below a rate of convergence in a misspecified setting by assuming a certain rate of decay for the separation radius as the number of design points increases. 
\begin{corollary} \label{coro:sob_sepa}
Let $\Omega, P, r, s, k_r, \H_{k_r}(\Omega)$ be the same as in Theorem \ref{theo:sob_sepa}. 
Suppose $\{ (w_i,X_i) \}_{i=1}^n \in (\R \times \Omega)^n$ is design points such that $e_n(P;\H_{k_r}(\Omega)) = O(n^{-b})$ and $q_{X^n} = \Theta(n^{-a})$ for some $b>0$ and $a > 0$, respectively, as $n \to \infty$.
Then for any $f \in C_B^s (\Omega) \cap H^s (\Omega)$, we have
\begin{equation} \label{eq:rate_sep_gen}
| P_n f - P f | = O(n^{- \min( b - a(r-s), as)} ) \quad (n \to \infty).
\end{equation}
In particular, the rate in the right hand side is optimized when $a = b/r$, which gives 
\begin{equation} \label{eq:rate_sobs}
| P_n f - P f | = O(n^{-\frac{bs}{r}}) \quad (n \to \infty).
\end{equation}
\end{corollary}
\begin{proof}
Plugging $e_n(P;\H_{k_r}(\Omega)) = O(n^{-b})$ and $q_{X^n} = \Theta(n^{-a})$ into \eqref{eq:bound_sob_sepa_379} yields 
\begin{eqnarray*}
\left| P_n f - Pf \right| 
&= & O(n^{a (r-s) -b }) + O(n^{-as})= O(n^{- \min( b - a(r-s), as)} ),
\end{eqnarray*}
which proves (\ref{eq:rate_sep_gen}).  The second assertion is obvious. \qed 
\end{proof}

\begin{remark} 
As stated in the assertion, the best rate for the bound is achieved when $a = b/r$.
The resulting rate in (\ref{eq:rate_sobs}) coincides with that of Corollary \ref{coro:rate_weight} (see (\ref{eq:rate_sobw})) with $c = 0$.
Therefore observations similar to those for Theorem \ref{theo:sob_weight} can be made with the rate in (\ref{eq:rate_sobs}).
\end{remark}

\section{Bayesian quadrature in misspecified settings}\label{sec:BQ_quasi}

To demonstrate the results of Section \ref{sec:sobolev}, a rate of convergence for Bayesian quadrature in misspecified settings is derived.
To this end, an upper-bound on the integration error of Bayesian quadrature is first provided, when the smoothness of an integrand is overestimated.
It is obtained by combining Theorem \ref{theo:sob_sepa} in Section \ref{sec:sobolev} and Proposition \ref{prop:BQ_fill} in Section \ref{sec:BQ_well}.

\begin{theorem} \label{theo:BQ_misspecified}
Let $\Omega \subset \R^d$ be a bounded open set with ${\rm diam}(\Omega) \leq 1$ such that an interior cone condition is satisfied and the boundary is Lipschitz, $P$ be a probability distribution on $\R^d$ with a bounded density function $p$ such that ${\rm supp}(P) \subset \Omega$, $r$ be a real number with $\lfloor r \rfloor > d/2$, and $s$ be a natural number with $s \leq r$.
Suppose that $k_r$ is a kernel on $\R^d$ satisfying Assumption \ref{assumption:kernel},  $X^n := \{X_1,\dots,X_n\} \subset \Omega$ is design points such that $G := (k_r(X_i,X_j))_{i,j=1}^n \in \R^{n \times n}$ is invertible, and $w_1,\dots,w_n$ are the Bayesian quadrature weights in (\ref{eq:BQ_weight}) based on $k_r$. 
Assume that there exist constants $c_q > 0$ and $\delta > 0$ independent of $X^n$, such that $1- s/r < \delta \leq 1$ and 
\begin{equation} \label{eq:BQ_bound_cond_31}
h_{X^n,\Omega} \leq c_q q_{X^n}^\delta.
\end{equation} 
Then there exist positive constants $C$ and $h_0$ independent of $X^n$, such that for any $f \in C_B^s (\Omega) \cap H^s (\Omega)$, we have
\begin{equation} \label{eq:Bq_rate_sepa}
\left| P_n f - Pf \right| \leq C  \max \left(  \| f \|_{C_B^s(\Omega)}, \| f \|_{H^s(\Omega)}  \right)  h_{X^n,\Omega}^{r - (r-s)/\delta},
\end{equation}
provided that $h_{X^n,\Omega} \leq h_0$.
\end{theorem}
\begin{proof}
Under the assumptions, Theorem \ref{theo:sob_sepa} gives that
\begin{equation} \label{eq:Theorem1_BQ_48}
\left| P_n f - Pf \right| \leq C_1 \max \left(  \| f \|_{C_B^s(\Omega)}, \| f \|_{H^s(\Omega)}  \right)  \left(  q_{X^n}^{- (r-s) } e_n(P; \H_{k_r}(\Omega)) + q_{X^n}^{s} \right),
\end{equation}
where $C_1 > 0$ is a constant, and $e_n(P; \H_{k_r}(\Omega))$ is the worst case error of $\{ (w_i,X_i) \}_{i=1}^n$ in $\H_{k_r}(\Omega)$. On the other hand, Proposition \ref{prop:BQ_fill} implies that there exist constants $C_2 > 0$ and $h_0 > 0$ independent of the choice of $X^n$, such that 
\begin{equation} \label{eq:Propsition1_BQ_54}
e_n(P; \H_{k_r}(\Omega))  \leq C_2 h_{X^n,\Omega}^r,
\end{equation}
provided that $h_{X^n,\Omega} \leq h_0$. 
Note also that (\ref{eq:BQ_bound_cond_31}) implies that 
\begin{equation} \label{eq:q_x_bound_32}
q_{X^n}^{-1} \leq c_q^{1/\delta} h_{X^n, \Omega}^{-1/\delta}.
\end{equation}
From $q_{X^n} \leq h_{X^n,\Omega}$ and the above inequalities, it follows that 
\begin{eqnarray*}
\left| P_n f - Pf \right| 
&\stackrel{\eqref{eq:Theorem1_BQ_48}}{\leq}&  C_1 \max \left(  \| f \|_{C_B^s(\Omega)}, \| f \|_{H^s(\Omega)}  \right)  \left(  q_{X^n}^{- (r-s) } e_n(P; \H_{k_r}(\Omega)) + q_{X^n}^{s} \right)\\
&\stackrel{\eqref{eq:Propsition1_BQ_54}}{\leq}&  C_1 \max \left(  \| f \|_{C_B^s(\Omega)}, \| f \|_{H^s(\Omega)}  \right)  \left( C_2  q_{X^n}^{- (r-s) }  h_{X^n,\Omega}^r + q_{X^n}^{s} \right)\\
&\stackrel{\eqref{eq:q_x_bound_32}}{\leq}&  C_1 \max \left(  \| f \|_{C_B^s(\Omega)}, \| f \|_{H^s(\Omega)}  \right)  \left( C_2 c_q^{(r-s)/\delta} h_{X^n,\Omega}^{r- (r-s)/\delta} + q_{X^n}^{s} \right)  \\
&\stackrel{(\star)}{\leq}&  C_1 \max \left(  \| f \|_{C_B^s(\Omega)}, \| f \|_{H^s(\Omega)}  \right)  \left( C_2 c_q^{(r-s)/\delta} h_{X^n,\Omega}^{r- (r-s)/\delta} + h_{X^n}^{s} \right)  \\
&\stackrel{(\dagger)}{\leq}& C_3 \max \left(  \| f \|_{C_B^s(\Omega)}, \| f \|_{H^s(\Omega)}  \right) h_{X^n,\Omega}^{r- (r-s)/\delta},
\end{eqnarray*}
where $C_1$, $C_2$ and $C_3$ are positive constants independent of the choice of design points $X^n$, and we used $q_{X^n} \leq h_{X^n,\Omega}$ in $(\star)$, $ 0 < h_{X^n} \leq 1$ and $0 < r - (r-s)/\delta \leq s$ in $(\dagger)$.  \qed
\end{proof}

\begin{remark} \rm
\begin{itemize}
\item The condition (\ref{eq:BQ_bound_cond_31}) implies that
\begin{equation} \label{eq:BQ_bound_remark_50}
c' h_{X^n,\Omega}^{1/\delta} \leq q_{X^n} \leq h_{X^n,\Omega},
\end{equation}
where $c' := c_q^{-1/\delta}$ is independent of $X^n$.
This condition is stronger for a larger value of $\delta$, requiring that distinct design points should not be very close to each other.
Note that the lower-bound $1-s/r<\delta$ is necessary for the upper-bound of the error (\ref{eq:Bq_rate_sepa}) to have a positive exponent, while the upper-bound $\delta \leq 1$ follows from $q_{X^n} \leq h_{X^n,\Omega}$, which holds by definition.
The constraint $1-s/r<\delta$ and (\ref{eq:BQ_bound_remark_50}) thus imply that a stronger condition is required for $X^n$ as the degree of misspecification becomes more serious (i.e., as the ratio $s/r$ becomes smaller).\vspace{-1mm}

\item 
If the condition (\ref{eq:BQ_bound_cond_31}) is satisfied for $\delta = 1$, then the design points $X^n$ are called {\em quasi-uniform} \cite[Section 7.3]{SchWen06}. 
In this case, the bound in (\ref{eq:Bq_rate_sepa}) is  
\begin{equation} \label{eq:bound_quasi_131}
| P_n f - Pf | \leq C  \max \left(  \| f \|_{C_B^s(\Omega)}, \| f \|_{H^s(\Omega)}  \right)  h_{X,\Omega}^s.
\end{equation}
This is the same order of approximation as that of Proposition \ref{prop:BQ_fill} when $r = s$.
Proposition \ref{prop:BQ_fill} provides an error bound for Bayesian quadrature in a well-specified case, where one knows the degree of smoothness $s$ of the integrand. 
Therefore, (\ref{eq:bound_quasi_131}) suggests that, if the design points are quasi-uniform, then Bayesian quadrature can be adaptive to the (unknown) degree of the smoothness $s$ of the integrand $f$, even in a situation where one only knows its upper-bound $r \geq s$.

\end{itemize}
\end{remark}

We obtain the following as a corollary of Theorem \ref{theo:BQ_misspecified}. The proof is obvious, and omitted. 
\begin{corollary} \label{coro:BQ_misspecified_rate}
Let $\Omega, P, r, s, k_r, X^n, G$ and $w_i$ ($i=1,\ldots,n$) be the same as Theorem \ref{theo:BQ_misspecified}. 
Assume that there exist constants $c_q > 0$ and $\delta > 0$ independent of $X^n$, such that $1- s/r < \delta \leq 1$ and
\[
h_{X^n,\Omega} \leq c_q q_{X^n}^\delta, 
\]
and further $h_{X^n,\Omega} = O(n^{- \alpha})$ as $n \to \infty$ for some $0 < \alpha \leq 1/d$. 
Then for all $f \in C_B^s (\Omega) \cap H^s (\Omega)$, we have
\begin{equation} \label{eq:BQ_misspecified_rate_gen}
\left| P_n f - Pf \right| = O(n^{  - \alpha [ r - (r-s)/\delta ]  }) \quad (n \to \infty).
\end{equation}
In particular, the best possible rate in the right hand side is achieved when $\delta = 1$ and $\alpha = 1/d$, giving that
\begin{equation} \label{eq:BQ_misspecified_rate}
\left| P_n f - Pf \right| = O(n^{  - s/d }) \quad (n \to \infty).
\end{equation}
\end{corollary}
\begin{remark} \rm
\begin{itemize}
\item 
The rate $O(n^{-s/d})$ in (\ref{eq:BQ_misspecified_rate}) matches the minimax optimal rate of deterministic quadrature rules for the worst case error in the Sobolev space $H^s(\Omega)$ with $\Omega$ being a cube \citep[Proposition 1 in Section 1.3.12]{Nov88}. 
Therefore, it is shown that the optimal rate may be achieved by Bayesian quadrature, even in the misspecified setting (under a slightly stronger assumption that $f \in H^s(\Omega) \cap C_B^s(\Omega)$).
In other words, Bayesian quadrature may achieve the optimal rate {\em adaptively}, without knowing the degree $s$ of smoothness of a test function: one just needs to know its upper bound $r \geq s$.\vspace{-1mm}

\item The main assumptions required for the optimal rate (\ref{eq:BQ_misspecified_rate}) are that (i) $h_{X^n,\Omega} = O(n^{-1/d})$ and that (ii) $h_{X^n,\Omega} \leq c_q q_{X^n}^\delta$ for $\delta = 1$.
Recall that (i) is the same assumption that is required for the optimal rate $O(n^{-r/d})$ in the well-specified setting $f \in H^r(\Omega)$ (Corollary \ref{coro:BQ_rate_well}). 
On the other hand, (ii) is the one required for the finite sample bound in Theorem \ref{theo:BQ_misspecified}. 
Both these assumptions are satisfied, for instance, if $X_1,\dots,X_n$ are grid points in $\Omega$. 
\end{itemize}
\end{remark}

\section{Simulation experiments} \label{sec:experiments}
We conducted simulation experiments to empirically assess the obtained theoretical results.
MATLAB code for reproducing the results is available at \verb|https://github.com/motonobuk/kernel-quadrature|.
We focus on Bayesian quadrature in these experiments.
\subsection{Problem setting} \label{sec:setting-simulation}

\paragraph{Domain, distribution and design points.}
The domain is $\Omega := [0,1] \subset \R$ and the measure of quadrature $P$ is the uniform distribution over $[0,1]$.
For design points, we consider the following two configurations:
\begin{itemize}
\item \textbf{Uniform}: $X^n = \{X_1,\dots,X_n\}$ are equally-spaced grid points in $[0,1]$ with $X_1 = 0$ and $X_n = 1$, that is, $X_i = (i-1) / (n-1)$ for $i = 1,\dots,n$. \vspace{-1mm}
\item \textbf{Non-uniform}: $X^n = \{X_1,\dots,X_n \}$ are non-equally spaced points in $[0,1]$, such that $X_i = (i-1)/(n-1)$ if $i$ is odd, and $X_i = X_{i-1} + (n-1)^{-2}$ if $i$ is even.
\end{itemize}
For the {\em uniform} design points, both the fill distance $h_{X^n,\Omega}$ and the separation radius $q_{X^n, \Omega}$ decay at the rate $O(n^{-1})$.
On the other hand, for the {\em non-uniform} points the separation radius decays at the rate $O(n^{-2})$, while the rate of the fill distance remains the same $O(n^{-1})$ as for the uniform points.
Using these two different sets of design points, we can observe the effect of the separation radius to the performance of kernel quadrature.\vspace{-3mm}

\paragraph{Kernels.}
As before, $r$ denotes the assumed degree of smoothness used for computing quadrature weights, and $s$ denotes the true smoothness of test integrands, both expressed in terms of Sobolev spaces.
As kernels of the corresponding Sobolev spaces, we used Wendland kernels \cite[Definition 9.11]{Wen05}, which are given as follows \cite[Corollary 9.14]{Wen05}. Define the following univariate functions:
\begin{eqnarray*}
\phi_{1,0}(t) &:=& (1-t)_+, \quad \phi_{1,1}(t) := (1-t)_+^3 ( 3 t + 1 ), \\
\phi_{1,2}(t) &:=& (1-t)_+^5 ( 24 t + 15 t + 3), \\
\phi_{1,3}(t) &:=& (1-t)_+^7 ( 315 t^3 + 285 t^2 + 105 t + 15), \quad t \geq 0,
\end{eqnarray*}
where $(x)_+ := \min(0,x)$.
The Wendland kernel $k_r$ whose RKHS is norm-equivalent to the Sobolev space $H^r ([0,1])$ of order $r\ (= 1,2,3,4)$ is then defined by
$k_r(x,y) := \phi_{d,r-1}( | x - y | / \delta )$ for $x, y \in [0,1]$ \citep[Theorem 10.35]{Wen05},
where $\delta$ is a scale parameter and we set it to be $0.1$.\vspace{-3mm}

\paragraph{Evaluation measure.}
For each pair of  $r\ (= 1,2,3,4)$ and $s\ (= 1,2,3,4)$, we first computed quadrature weights $w_1,\dots,w_n$ by minimizing the worst case error in $H^r([0,1])$, and then evaluated the quadrature rule $(w_i,X_i)_{i=1}^n$ by computing the worst case error in $H^s([0,1])$, that is, $\sup_{\| f \|_{H^s([0,1])} \leq 1} |P_n f- Pf|$.
More concretely, we computed the weights $w_1,\dots,w_n$ by the formula \eqref{eq:BQ_weight} for Bayesian quadrature using the kernel $k_r$, and then evaluated the worst case error \eqref{eq:wce_preliminary} by computing the square root of \eqref{eq:kmean_analytic} using the kernel $k_s$.
In this way, one can evaluate the performance of kernel quadrature under various settings.
For instance, the case $s < r$ is a situation where the true smoothness $s$ is smaller than the assumed one $r$, the misspecified setting we have dealt in the paper.

\subsection{Results}
The simulation results are shown in Figure \ref{fig:simulation-equi} (\textbf{Uniform} design points) and Figure \ref{fig:simulation-irreg} (\textbf{Non-uniform} design points).
In the figures, we also report the exponents in the empirical rates of the fill distance $h_{X^n,\Omega}$, the separation radius $q_{X^n}$ and the absolute sum of weights $\sum_{i=1}^n |w_i|$ in the top of each subfigure; see the captions of Figures \ref{fig:simulation-equi} and \ref{fig:simulation-irreg} for details. Based on these, we can draw the following observations.

\paragraph{Optimal rates in the well-specified case.}
In both Figures \ref{fig:simulation-equi} and \ref{fig:simulation-irreg}, the black solid lines are the worst case errors in the well specified case $s = r$.
The empirical convergence rates of these worst case errors are very close to the optimal rates derived in Section \ref{sec:BQ_well} (see Corollary \ref{coro:BQ_rate_well} and its remarks), confirming the theoretical results.
Proposition \ref{prop:BQ_fill} and Corollary \ref{coro:BQ_rate_well} also show that the worst case error in the well-specified case is determined by the fill distance and is independent of the separation radius.
The simulation results are consistent with this, since for both Figures \ref{fig:simulation-equi} and \ref{fig:simulation-irreg} the fill distance decays essentially at the rate $O(n^{-1})$, while the separation radius decays quicker for Figure \ref{fig:simulation-irreg} than for Figure \ref{fig:simulation-equi}.

\begin{figure}[th]
 \begin{center}
 \subfigure[$r=1$]{
 \includegraphics[clip,width=55mm]{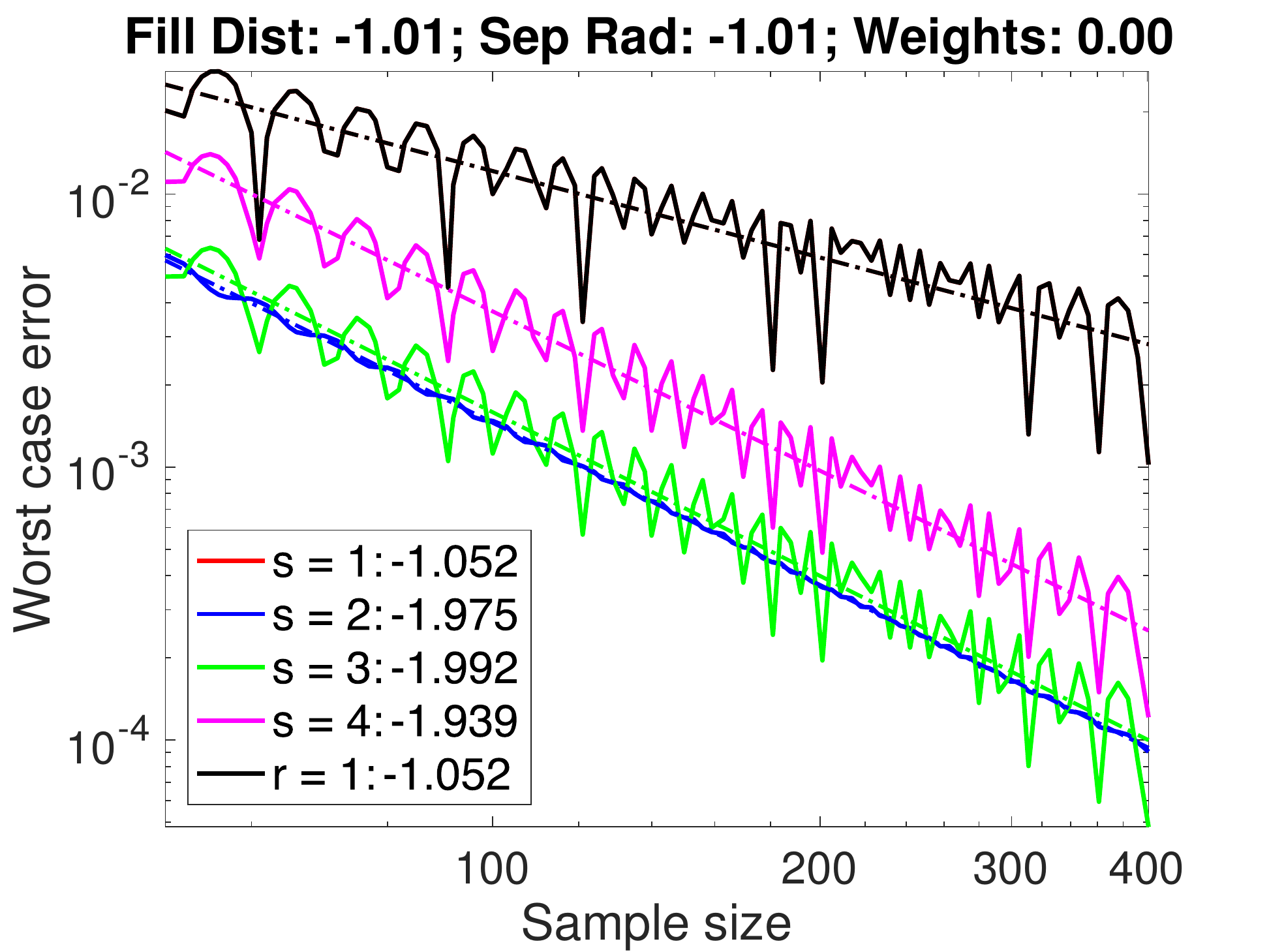}
 \label{fig:eq_r1}}
 \subfigure[$r=2$]{
  \includegraphics[clip,width=55mm]{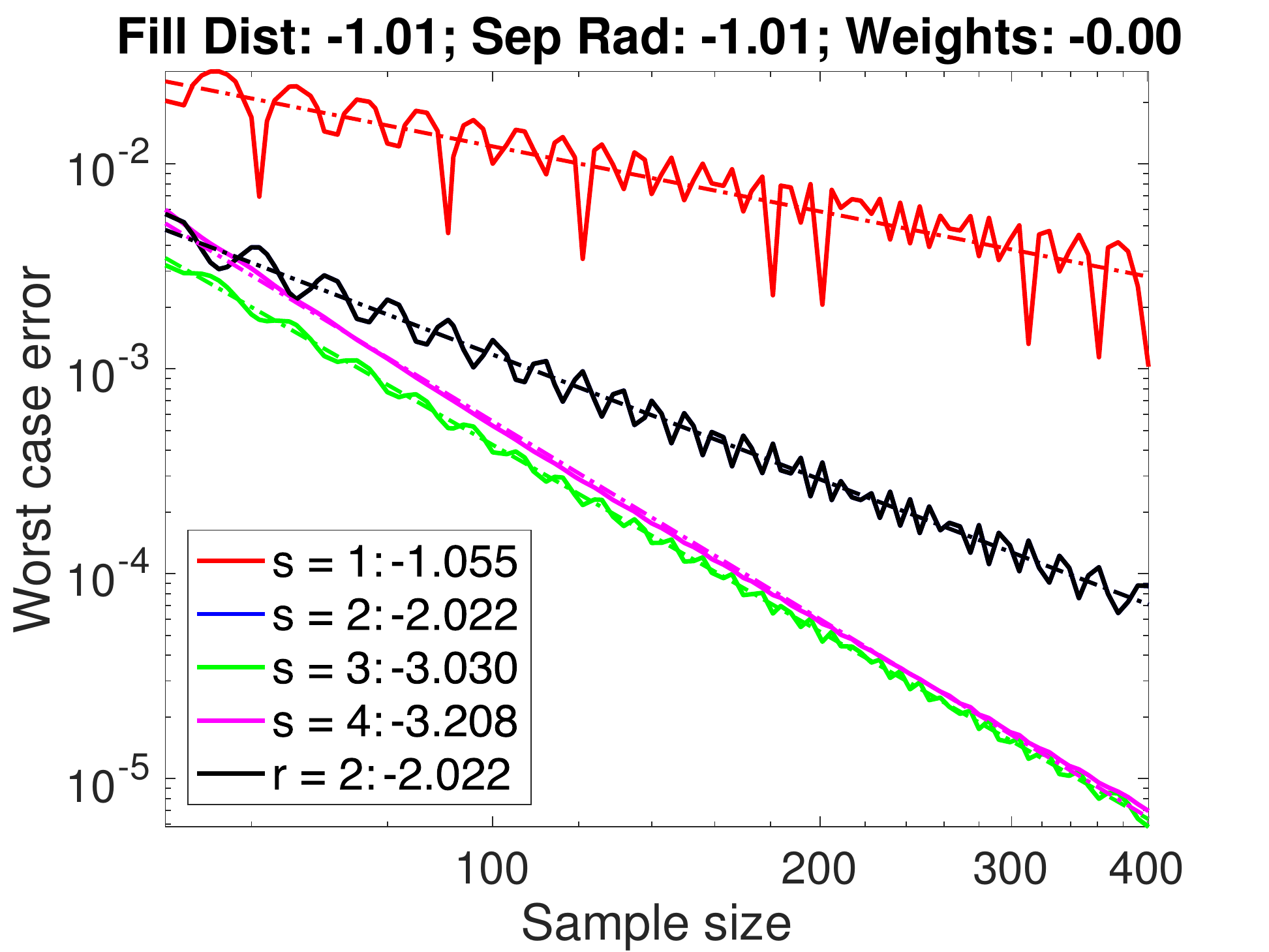}}
  \subfigure[$r=3$]{
    \includegraphics[clip,width=55mm]{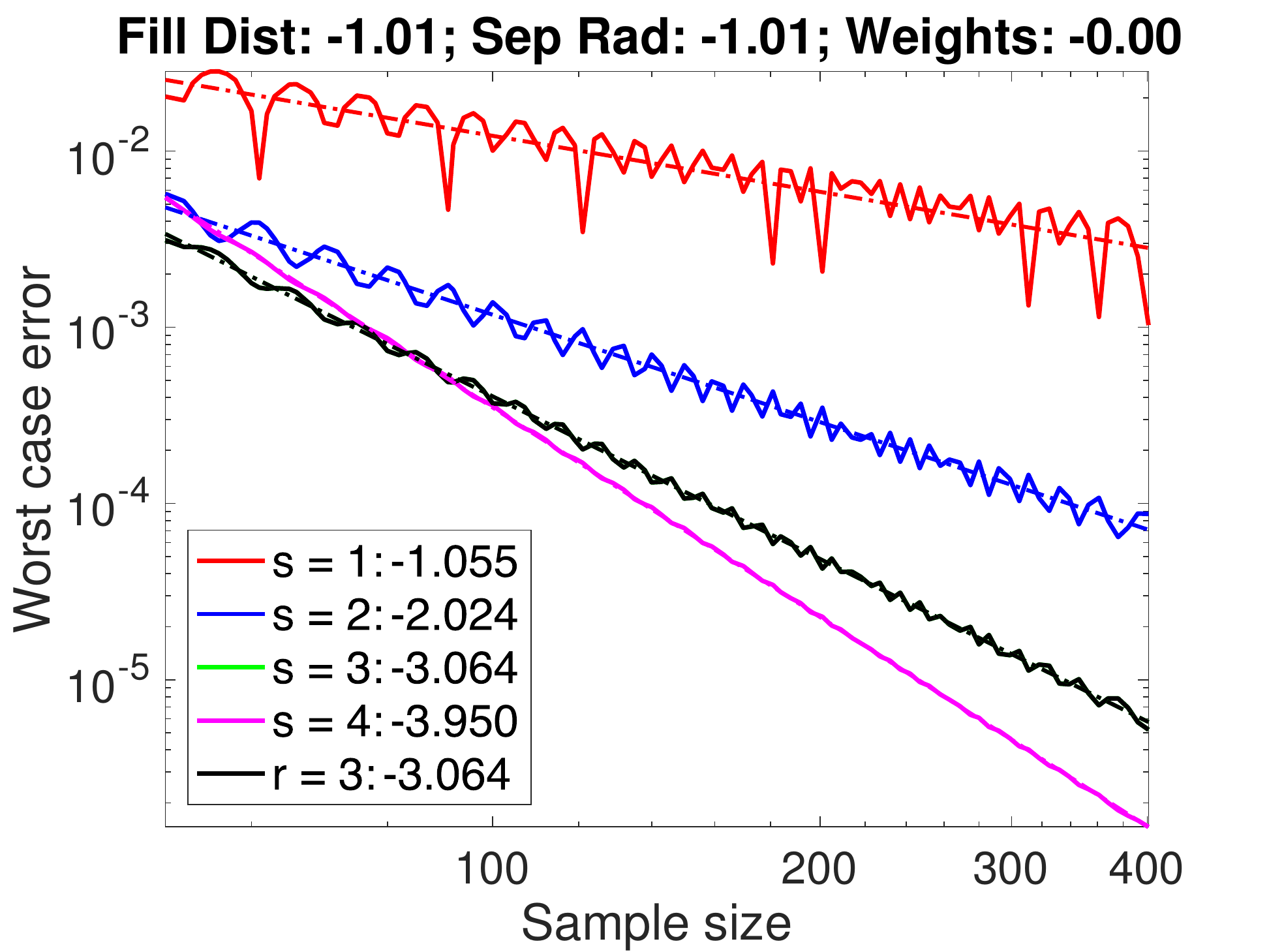}}
    \subfigure[$r=4$]{
      \includegraphics[clip,width=55mm]{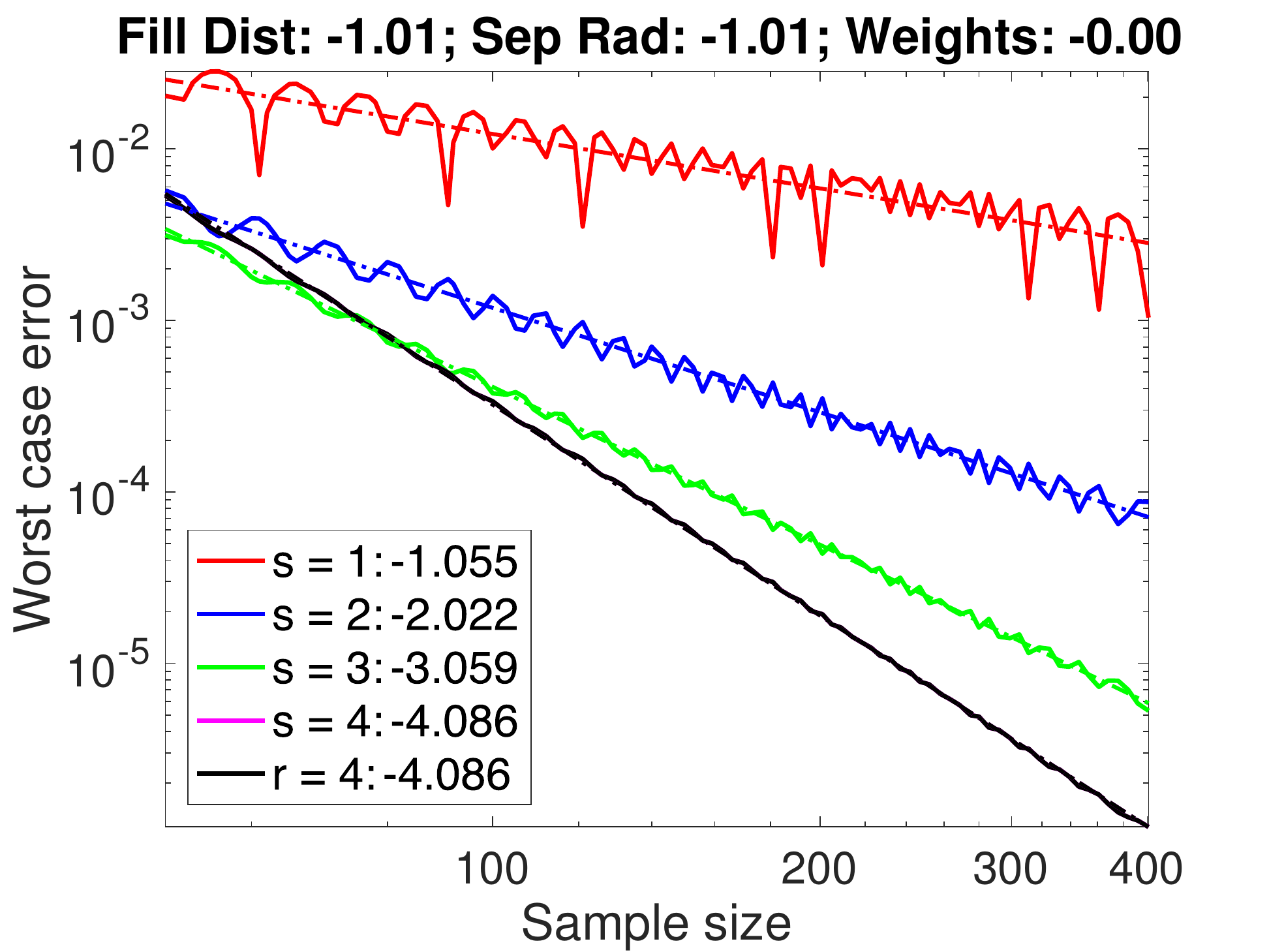}}
 \end{center}
\vspace{-4mm}
 \caption{
 Design points are \textbf{Uniform}, i.e., equally-spaced grid points in $[0,1]$; see Section \ref{sec:setting-simulation} for details.
 The solid lines are the worst case errors and the dotted lines are the corresponding linear fits.
 The subfigures (a)--(d) are respectively the results for the weights computed using the kernel $k_r$ with $r = 1, 2, 3, 4$.
 Black lines are the worst case errors for the well-specified case $s= r$ (i.e., the worst case error is evaluated in the same Sobolev space where the weights are obtained).
   Note that black lines overlap the corresponding lines for $s = r$ (e.g., in the subfigure (a) the red line for $s=1$ does not appear since the black line completely overlaps it).
 In each legend, we report the exponents of the empirical rates of the worst case errors.
 For instance, in the subfigure (d), the worst case error for $s = 1$ decays at the rate $O(n^{-1.055})$. 
 On the top of each figure, the exponents in the empirical rates of the fill distance $h_{X^n,\Omega}$, the separation radius $q_{X^n}$ and the absolute sum of weights $\sum_{i=1}^n |w_i|$ are shown. For instance, for the subfigure (d), we have $h_{X^n,\Omega} = O(n^{-1.01})$, $q_{X^n} = O(n^{-1.01})$ and $\sum_{i=1}^n |w_i| = O(n^{0.00})$. 
 }
 \label{fig:simulation-equi}
\end{figure}
\vspace{-4mm}
\paragraph{Adaptability to lesser smoothness.}
Let us look at Figure \ref{fig:simulation-equi} for the misspecified case $s < r$, i.e., where the true smoothness $s$ is smaller than the assumed one $r$.
For every pair of $s < r$, the rates are very close to the optimal ones, showing that adaptation to the unknown lesser smoothness in fact occurs.
This is consistent with Corollaries \ref{coro:sob_sepa} and \ref{coro:BQ_misspecified_rate}, which imply that adaptation occurs if the design points are quasi-uniform.
Figure \ref{fig:simulation-irreg} shows also some adaptability, but the rates for $s = 1$ with $r > s$ are slower than the optimal one. 
This will be discussed below, in a discussion on the effect of the separation radius.

\begin{figure}[th]
 \begin{center}
    \subfigure[$r=1$]{
    \includegraphics[clip,width=55mm]{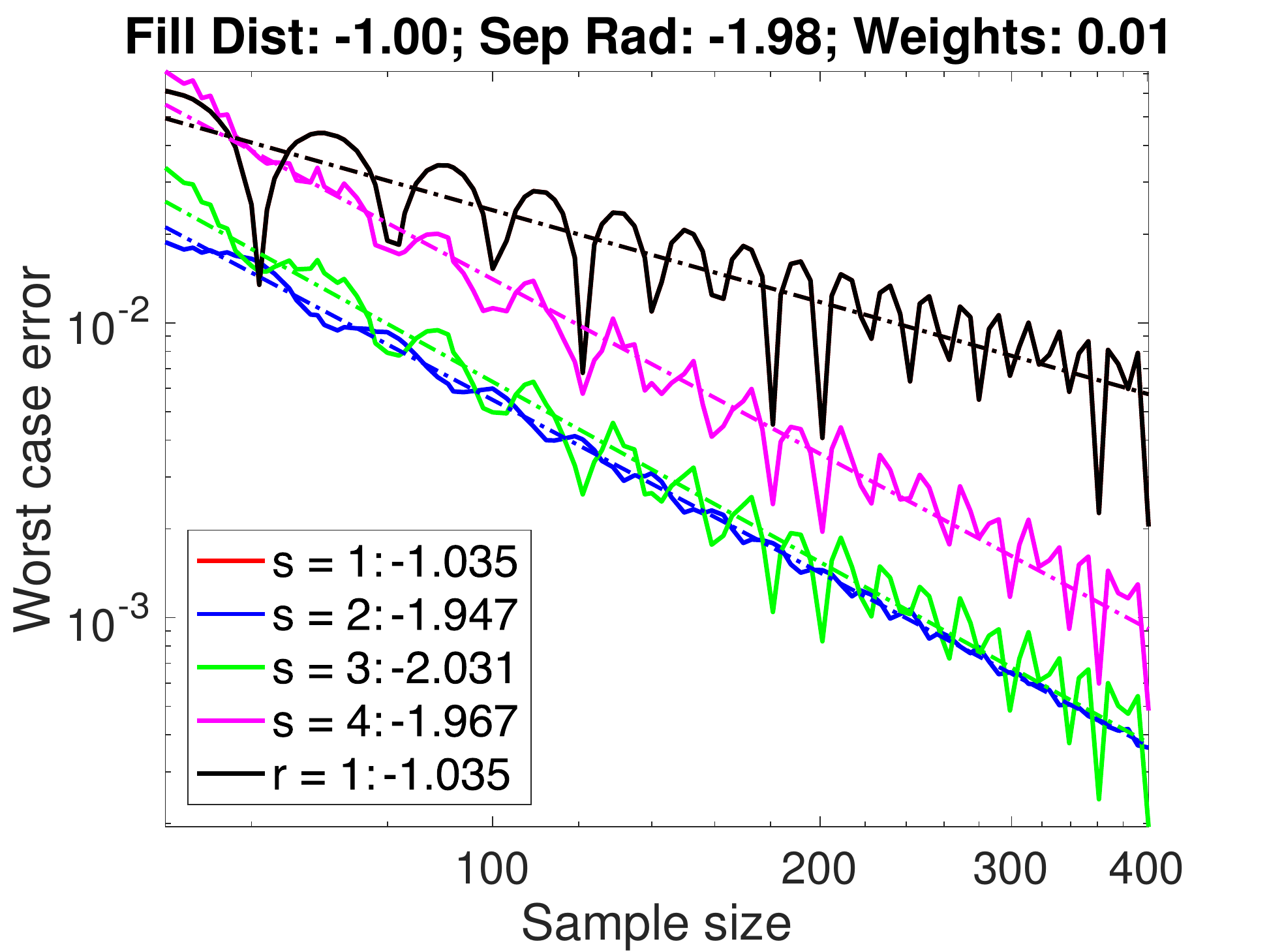}
    }
    \subfigure[$r=2$]{
    \includegraphics[clip,width=55mm]{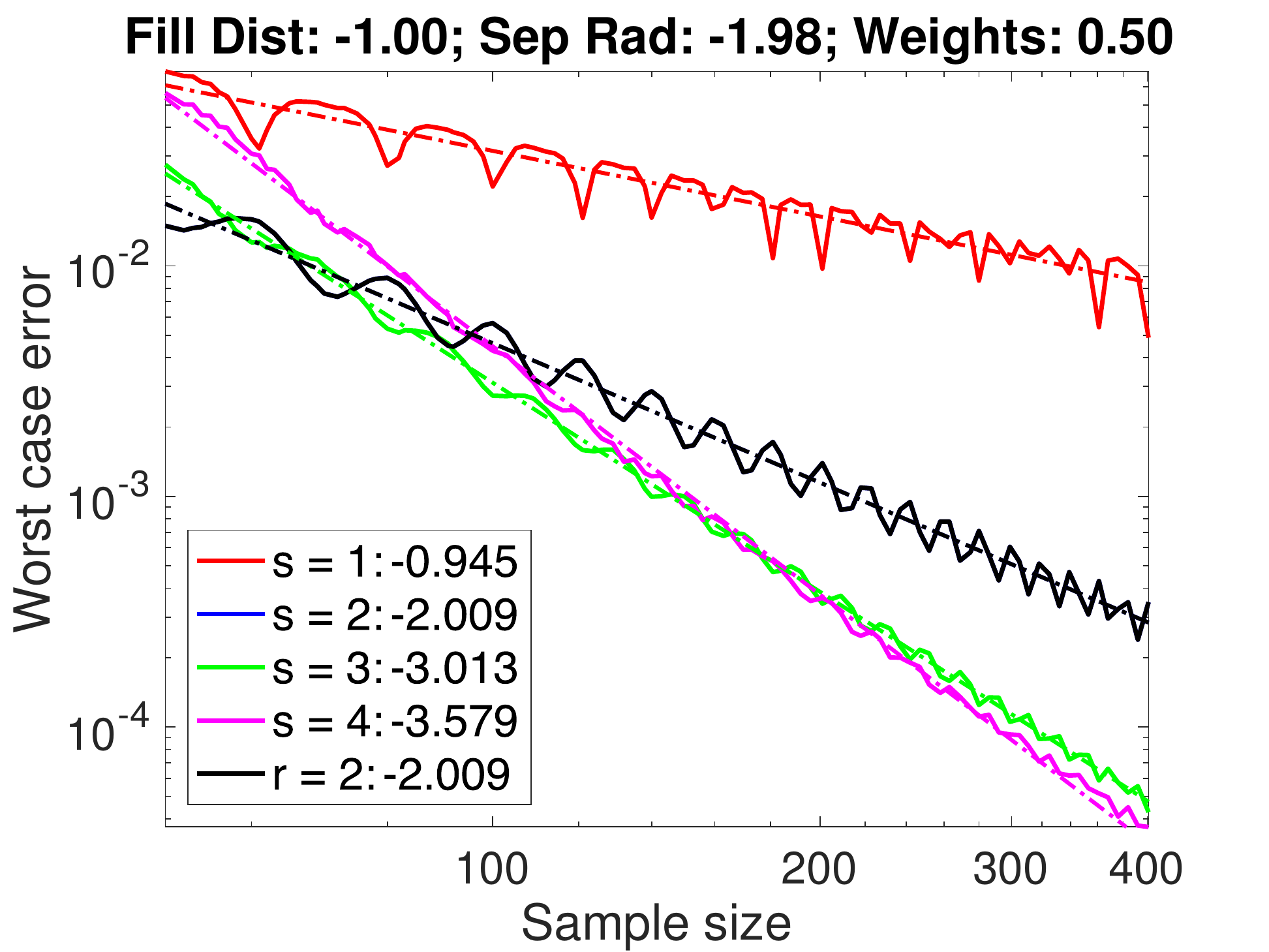}}
    \subfigure[$r=3$]{
     \includegraphics[clip,width=55mm]{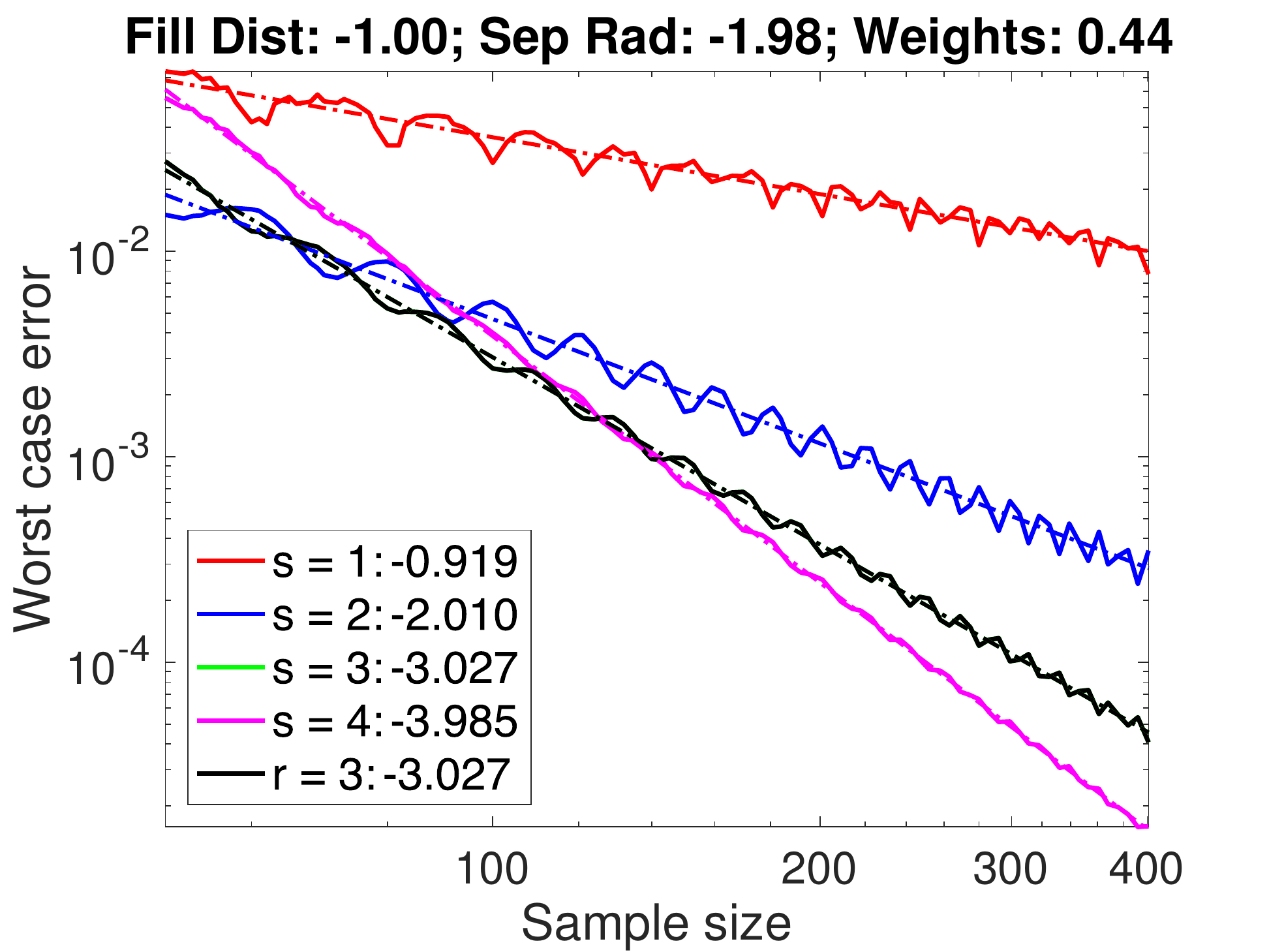}}
     \subfigure[$r=4$]{
    \includegraphics[clip,width=55mm]{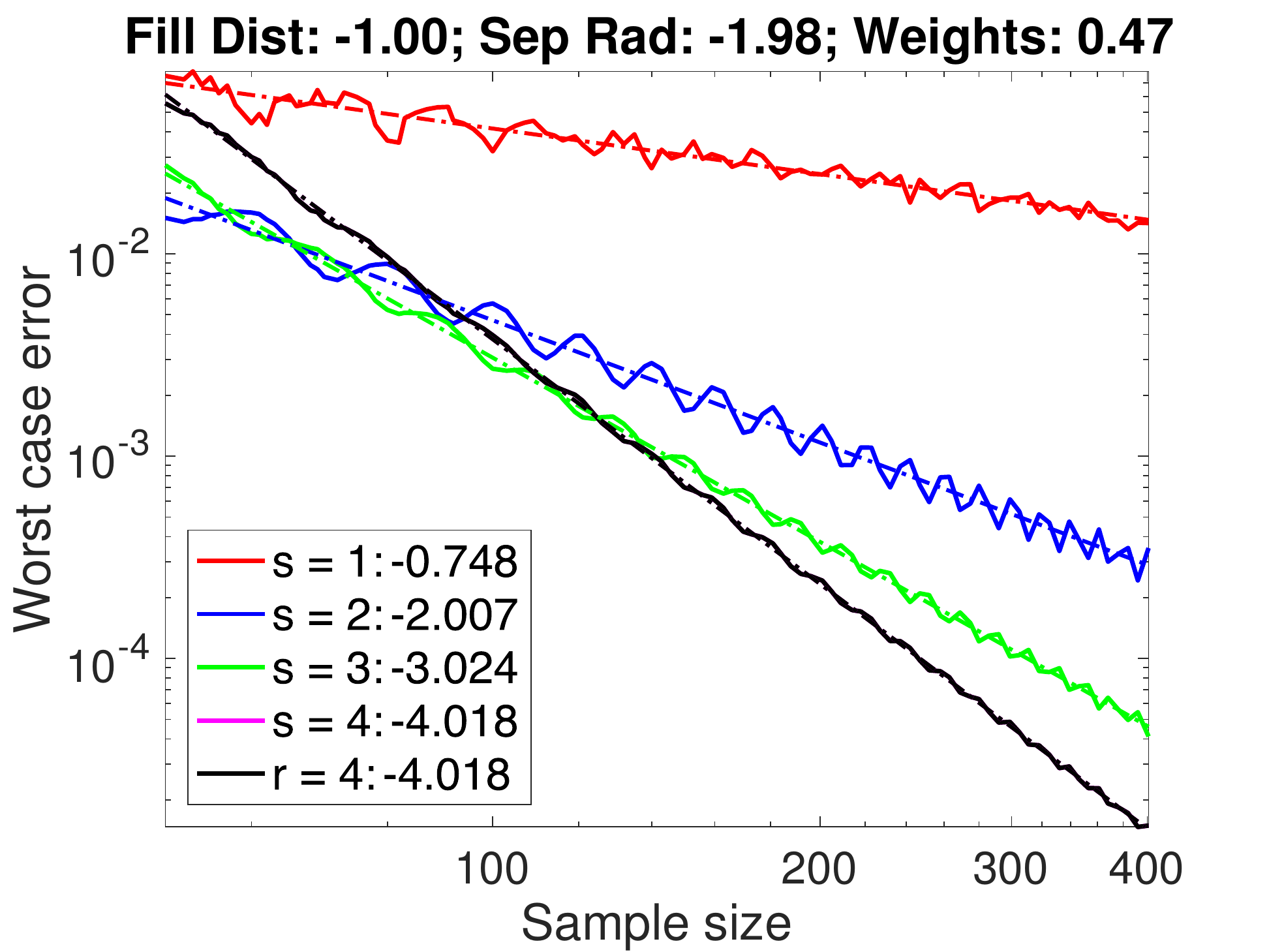}}
 \end{center}
\vspace{-4mm}
 \caption{
 Design points are \textbf{Non-uniform}, i.e., non-equally spaced points in $[0,1]$; see Section \ref{sec:setting-simulation} for details.
  The solid lines are the worst case errors and the dotted lines are the corresponding linear fits. 
   The subfigures (a)--(d) are respectively the results for the weights computed using the kernel $k_r$ with $r = 1, 2, 3, 4$.
   Black lines are the worst case errors for the well-specified case $s= r$ (i.e., the worst case error is evaluated in the same Sobolev space where the weights are obtained).
   Note that black lines overlap the corresponding lines for $s = r$ (e.g., in the subfigure (a) the red line for $s=1$ does not appear since the black line completely overlaps it).
   In each legend, we report the exponents of the empirical rates of the worst case errors.
 For instance, in the subfigure (d), the worst case error for $s = 1$ decays at the rate $O(n^{-0.748})$.
 On the top of each figure, the exponents in the empirical rates of the fill distance $h_{X^n,\Omega}$, the separation radius $q_{X^n}$ and the absolute sum of weights $\sum_{i=1}^n |w_i|$ are shown.
 For instance, for the subfigure (d), we have $h_{X^n,\Omega} = O(n^{-1.00})$, $q_{X^n} = O(n^{-1.98})$ and $\sum_{i=1}^n |w_i| = O(n^{0.47})$.
 }
 \label{fig:simulation-irreg}
\end{figure}
\vspace{-6mm}
\paragraph{Adaptability to greater smoothness.}
While the case $s > r$ is not covered by our theoretical analysis, Figures \ref{fig:simulation-equi} and \ref{fig:simulation-irreg} show some adaptation to the greater smoothness.
This phenomenon is also observed by Bach \cite[Section 5]{Bac17}, who showed (for quadrature weights obtained with {\em regularized} matrix inversion) that, if $2r \geq s > r$ then the optimal rate is still attainable in an adaptive way. 
Bach \cite[Section 6]{Bac17} verified this finding in experiments with quadrature weights {\em without} regularization. 
In our experiments, this phenomenon is observed for all cases of  $2 r \geq s > r$ expect for the case $r = 2$ and $s = 4$ in both Figures \ref{fig:simulation-equi} and \ref{fig:simulation-irreg}.
Note however that in \cite{Bac17}, design points are assumed to be {\em randomly} generated from a specific proposal distribution, so the results there are not directly applicable to deterministic quadrature rules.

\paragraph{The effect of the separation radius.}
In Figure \ref{fig:simulation-equi}, the rate for $s = 1$, that is $O(n^{-1.052})$, remains essentially the same for different values of $r = 1,2,3,4$.
This rate is essentially the optimal rate for $s = 1$, thus showing the adaptability of Bayesian quadrature to the unknown lesser smoothness (for $r = 2, 3, 4$).
On the other hand, in Figure \ref{fig:simulation-irreg} on non-uniform design points, the rate for $s = 1$ becomes slower as $r$  increases.
That is, the rates are $O(n^{-1.035})$ for $r = 1$ (the well-specified case), $O(n^{-0.945})$ for $r = 2$, $O(n^{-0.919})$ for $r = 3$ and $O(n^{-0.748})$ for $r = 4$.
This phenomenon may be attributed to the fact that the separation radius of the design points for Figure \ref{fig:simulation-irreg} decays faster than those for Figure \ref{fig:simulation-equi}. Corollary \ref{coro:BQ_misspecified_rate} shows that the rates in the misspecified case $s < r$ become slower as the separation radius decays more quickly and/or as the gap $r-s$ (or the degree of misspecification) increases, and this is consistent with the simulation results.
\vspace{-4mm}

\paragraph{The effect of the weights.}
While the sum of absolute weights $\sum_{i=1}^n |w_i|$ remains constant in Figure \ref{fig:simulation-equi}, this quantity increases in Figure \ref{fig:simulation-irreg}.
In the notation of Corollary \ref{coro:rate_weight}, $\sum_{i=1}^n |w_i| = O(n^c)$ with $c = 0$ for Figure \ref{fig:simulation-equi} while $c \approx 0.5$ for Figure \ref{fig:simulation-irreg} with $r = 2, 3, 4$.
Therefore the observation given in the preceding paragraph is also consistent with Corollary \ref{coro:rate_weight}, since it states that larger $c$ makes the rates slower in the misspecified case.
Note that the separation radius and the quantity $\sum_{i=1}^n |w_i|$ is intimately  related in the case of Bayesian quadrature, since the weights are computed from the inverse of the kernel matrix as \eqref{eq:BQ_weight} and thus affected by the smallest eigenvalue of the kernel matrix, while this smallest eigenvalue strongly depends on the separation radius and the smoothness of the kernel; see e.g., \cite{Sch95}  \cite[Section 12]{Wen05} and references therein.

\section{Discussion}
\label{sec:conclude}
In this paper, we have discussed the convergence properties of kernel quadrature rules with deterministic design points in misspecified settings.
In particular, we have focused on settings where quadrature weighted points are generated based on misspecified assumptions on the degree of smoothness, that is, the situation where the integrand is less smooth than assumed.

We have revealed conditions for quadrature rules under which adaptation to the unknown lesser degree of smoothness occurs. 
In particular we have shown that a kernel quadrature rule is adaptive if the sum of absolute weights remains constant, or if the spacing between design points is not too small (as measured by the separation radius).
Moreover, by focusing on Bayesian quadratures as working examples, we have shown that they can achieve minimax optimal rates of the unknown degree of smoothness, if the design points are quasi-uniform.
We expect that this result provides a practical guide for developing kernel quadratures that are robust to the misspecification of the degree of smoothness; such robustness is important in modern applications of quadrature methods, such as numerical integration in sophisticated Bayesian models, since they typically involve complicated or black box integrands and thus misspecification is likely to happen.\vspace{1mm}

There are several important topics to be investigated as part of future work.\vspace{-4mm}

\paragraph{Other RKHSs.}
This paper has dealt with Sobolev spaces as RKHSs of kernel quadrature.
However, there are many other important RKHSs of interest where similar investigation can be carried out. 
For instance, Gaussian RKHSs (i.e. the RKHSs of Gaussian kernels) have been widely used in the literature on Bayesian quadrature.
Such an RKHS consists of functions with infinite degree of smoothness.
This makes theoretical analysis challenging: our analysis relies on the approximation theory developed by Narcowich and Ward \cite{NarWar04}, which only applies to the standard Sobolev spaces.
Similarly, the theory of \cite{NarWar04} is also not applicable to Sobolev spaces with dominating mixed smoothness, which have been popular in the QMC literature. 
In order to analyze quadrature rules in these RKHSs, we therefore need to extend the approximation theory of \cite{NarWar04} to such spaces. 
Overall, this is an important but challenging theoretical problem.
(We also mention that relevant results are available in follow-up papers \cite{NarWarWen05,NarWarWen06}.
While these results do not directly provide the desired generalizations due to the same reasons mentioned above, these could still be potentially useful for our purpose.)\vspace{-4mm}

\paragraph{Sequential (adaptive) quadrature.}
Another important direction is the analysis for kernel quadratures that sequentially select design points. 
Such methods are also called {\em adaptive}, since the selection of the next point $X_{n+1}$ depends on the function values $f(X_1), \dots, f(X_n)$ of the already selected points $X_1,\dots,X_n$. 
Note that the adaptability here is different from that of the current paper where we used it in the context of adaptability of quadrature to unknown degree of smoothness.
For instance, the WSABI algorithm by \cite{gunter_sampling_2014} is an example of adaptive Bayesian quadrature which is considered as 
state-of-the-art for the application of Bayesian model evidence calculation.
Such adaptive methods have been known to be able to outperform non-adaptive methods in the following case: the hypothesis space is imbalanced or non-convex (see e.g.~Section 1 of \citealp{Nov16}).
In the worst case error, the hypothesis space is the unit ball in the RKHS $\H$, which is balanced and convex and so adaptation does {\em not} help. In fact, it is known that the optimal rate can be achieved without adaptation.
However, if the hypothesis space is imbalanced (i.e. $f$ being in the hypothesis space does {\em not} imply that $-f$ is in the hypothesis space), then adaptive methods may perform better.
For instance, the WSABI algorithm focuses on {\em non-negative} integrands, which means that the hypothesis is imbalanced and thus adaptive selection helps. Our analysis in this paper has focused on the worst case error defined by the unit ball in an RKHS, which is balanced and convex.
A future direction is thus to consider the setting of imbalanced or non-convex hypothesis spaces, such as the one consisting of non-negative functions, which will enable us to analyze the convergence behavior of sequential or adaptive Bayesian quadrature in misspecified settings.\vspace{-4mm}
\paragraph{Random design points.}
We have focused on {\em deterministic} quadrature rules in this paper.
In the literature, however, the use of {\em random} design points has also been popular.
For instance, the design points of Bayesian quadrature might be i.i.d.~with a certain proposal distribution or generated as an MCMC sequence.
Likewise, QMC methods usually apply randomization to deterministic design points.
Our forthcoming paper will deal with such situations and provide more general results than the current paper.

\section*{Acknowledgements}
MK and KF acknowledge support by MEXT Grant-in-Aid for Scientific Research on Innovative Areas (25120012). 
MK has also been supported in part by MEXT KAKENHI (17K12654) and the European Research Council (StG Project PANAMA).
BKS is partly supported by NSF-DMS-1713011. Most of this work was carried out when MK was a postdoc at the 
Institute of Statistical Mathematics, Tokyo.

\appendix

\def\thesection{Appendix \Alph{section}}

\section{Key results of Narcowich and Ward \cite{NarWar04}} \label{sec:results_NarWar04}
Here we review some key results from \cite{NarWar04}, which are needed in the proofs for our results.
One reason for including this is that a certain assumption about a function of interest, that is its integrability, is lacking in the results of  \cite{NarWar04}; see Remark \ref{rem:append-assumptions} for details. Therefore for the sake of completeness (as well as for the ease of the reader) we provide restatements of those results.

For $\sigma > 0$, below we denote by $\mathcal{B}_\sigma$ a subset of $L_2(\R^d)$ such that each $f \in \mathcal{B}_\sigma$ has a spectral density whose support is contained in the (closed) ball $B(0,\sigma)$ with radius $\sigma$, i.e.,
\[
\mathcal{B}_\sigma := \left\{ f \in L_2(\R^d):\ {\rm supp}(\hat{f}) \subset B(0,\sigma)  \right\}.
\]
This is a Paley-Weiner class of band-limited functions.
Thus the functions in $\mathcal{B}_\sigma$ are analytic (and thus they are continuous), and vanish at infinity. 
Therefore $\mathcal{B}_\sigma \subset L_2(\R^d) \cap C_0(\R^d)$.

The following theorem is a restatement of Theorem 3.5 of \cite{NarWar04}.

\begin{theorem} \label{theo:NarWar04_Th3_5}
Let $X^n := \{ X_1,\dots, X_n \} \subset  \R^d$ be $n$ distinct points with separation radius $q_{X^n} := \frac{1}{2} \min_{i \neq j} \| X_i - X_j \|$, such that ${\rm diam}(X^n) := \max_{i,j} \| X_i - X_j \| \leq 1$.
Let $\sigma > 0$ be a constant such that
\[
\sigma \geq \sigma_0 := \frac{24}{q_{X^n}} \left\{ \frac{\sqrt{\pi}}{3} \Gamma \left( \frac{d+2}{2} \right)  \right\}^{\frac{2}{d+2}}.
\]
Then for any $f \in C_0(\R^d) \cap L_2(\R^d)$, there exists $f_\sigma \in \mathcal{B}_\sigma$ that satisfies
\[
f(X_i)  =  f_\sigma(X_i), \quad i = 1,\dots,n,
\]
and 
$$\max\left( \| f - f_\sigma \|_{C_0(\R^d) },   \| f - f_\sigma \|_{L_2(\R^d)} \right)  
\leq C_d \inf_{g \in \mathcal{B}_\sigma} \max \left( \| f - g \|_{C_0(\R^d)}, \| f - g \|_{L_2(\R^d)} \right)$$
with $C_d:=5 + 2^{d+3}$.
\end{theorem}

In the above theorem, $f_\sigma$ is an interpolant of $f$ on $X^n$.
Thus the theorem guarantees that such a $f_\sigma$ can be taken as a band-limited function with a sufficiently large band-length $\sigma$.
More precisely, the lower bound $\sigma_0$ for $\sigma$ is proportional to the reciprocal of the separation radius $q_{X^n}$. 
This means that the band-length $\sigma$ should increase as the minimum distance between distinct design points decreases. 

The following proposition is a restatement of Proposition 3.7 of \cite{NarWar04}, which establishes an upper-bound on the $L_1$-error for the approximate function defined in (\ref{eq:approx_g})---see \ref{sec:sob_approx}.
\begin{proposition} \label{prop:NarWar04_Prop3_7}
Let $s\in\N$ and $\alpha \in \N_0^{d}$ be a multi-index such that $| \alpha | < s$.
Suppose $f \in C_0^s(\R^d) \cap H^s(\R^d) \cap L_1(\R^d)$ and $g_\sigma$ is the approximate function defined in (\ref{eq:approx_g}).
Then for any $\sigma > 0$,
\[
\| \partial^\alpha f - \partial^\alpha g_\sigma \|_{L_\infty(\R^d)} \leq C_{s-| \alpha |} \sigma^{ |\alpha| - s } \| f \|_{C_0^s(\R^d)},
\]
where $C_{k-| \alpha |} > 0$ is a constant depending only on the value of $k-|\alpha|$ and the function $\psi$ of Lemma \ref{lemma:existence} in \ref{sec:appendix_lemma}. 
\end{proposition}

The following theorem, which is Theorem 3.10 in \cite{NarWar04}, provides an upper-bound on the approximation error of the interpolant $f_\sigma$.

\begin{theorem} \label{theo:NarWar04_Th3_10}
Let $s\in\N$ and $\alpha \in \N_0^{d}$ be a multi-index such that $| \alpha | < s$.
Suppose $f \in C_0^s(\R^d) \cap H^s(\R^d) \cap L_1(\R^d)$, $f_\sigma$ is the interpolant from Theorem \ref{theo:NarWar04_Th3_5} with $\sigma > 0$ and $X^n := \{X_1,\dots, X_n \}$ satisfies the conditions in Theorem \ref{theo:NarWar04_Th3_5}.
Then there is a constant $C_{|\alpha|, s, d}$ that depends only on $|\alpha|$, $s$ and $d$ such that 
\[
\left\| \partial^\alpha f -  \partial^\alpha f_\sigma \right\|_{L_\infty(\R^d)} \leq C_{|\alpha|, s, d} \sigma^{|\alpha| - s} \max\left( \| f \|_{C_0^s(\R^d)}, \| f \|_{H^s(\R^d)} \right).
\]
\end{theorem}

The following proposition, which is Proposition 3.11 in \cite{NarWar04}, provides an upper-bound on a Sobolev norm of the interpolant $f_\sigma$.
\begin{proposition} \label{prop:NarWar04_Prop3_11}
Let $s\in\N$. 
Suppose $f \in C_0^s(\R^d) \cap H^s(\R^d) \cap L_1(\R^d)$, 
$f_\sigma$ is the interpolant from Theorem \ref{theo:NarWar04_Th3_5} with $\sigma > 0$ and $X^n := \{X_1,\dots, X_n \}$ satisfies the conditions in Theorem \ref{theo:NarWar04_Th3_5}.
Then there is a constant $C_{s,d}$ that depends only on $s$ and $d$ such that
\[
\| f_\sigma \|_{H^s(\R^d)} \leq C_{s,d} \max \left( \| f \|_{C_0^s(\R^d)}, \| f \|_{H^s(\R^d)} \right).
\]
\end{proposition}

\begin{remark} \rm \label{rem:append-assumptions}
We have the following comments on Propositions~\ref{prop:NarWar04_Prop3_7}, \ref{prop:NarWar04_Prop3_11} and Theorem~\ref{theo:NarWar04_Th3_10}.\vspace{-3mm}
\begin{itemize}
\item 
In the original statement of Proposition 3.7 in \cite{NarWar04}, the assumption $f \in L_1(\R^d)$ is missing. 
However, since this assumption is required for the function $g_\sigma$ to be well-defined (see Lemma \ref{lemma:g_sigma_fourier_416}), we have included it in Proposition~\ref{prop:NarWar04_Prop3_7}. Since Theorem 3.10 and Proposition 3.11 of \cite{NarWar04} depend on Proposition 3.7, we have included the assumption $f \in L_1(\R^d)$ in Theorem~\ref{theo:NarWar04_Th3_10} and Proposition~\ref{prop:NarWar04_Prop3_11}.\vspace{-1mm}

\item In the original statement of Proposition 3.11 in \cite{NarWar04}, the condition $\sigma \geq 1$ is required. 
This condition is implicitly satisfied by $\sigma$ in Proposition \ref{prop:NarWar04_Prop3_11} as the condition on $\sigma$ in Theorem \ref{theo:NarWar04_Th3_5} implies $\sigma\ge 1$, which can be seen from the fact that
$q_{X^n} \leq 1/2$ (follows from the assumption ${\rm diam}(X^n) \leq 1$) and the definition of the lower-bound $\sigma_0$ of $\sigma$. 



\end{itemize}
\end{remark}

\subsection{The Sobolev norm of the interpolant $f_\sigma$}


Here we provide an upper-bound on the Sobolev (RKHS) norm of the interpolant $f_\sigma$ in Theorem \ref{theo:NarWar04_Th3_5}.
The result essentially follows from an argument in p.298 of \cite{NarWar04}, but we prove it for completeness. 


\begin{lemma} \label{lemma:Sob_norm_interpolant_78}
Let $r\in\R$, $r > d/2$ and $s \in \N$, $r \geq s$. Let $k_r$ be a kernel on $\R^d$ such that $k_r(x,y) := \Phi(x-y)$, where $\Phi:\R^d \to \R$ satisfies 
\[
C_1 (1 + \| \xi \|^2)^{-r} \leq \hat{\Phi}(\xi), \quad \xi \in \R^d
\]
for some constant $C_1 > 0$ independent of $\xi$.
Suppose $f \in C_0^s(\R^d) \cap H^s(\R^d) \cap L_1(\R^d)$, $f_\sigma$ is the interpolant from Theorem \ref{theo:NarWar04_Th3_5} with $\sigma > 0$ and $X^n := \{X_1,\dots, X_n \}$ satisfies the conditions in Theorem \ref{theo:NarWar04_Th3_5}.
Then we have
\begin{equation}
\| f_{\sigma} \|_{ \H_{k_r} } \leq C_{s,d,k_r} \sigma^{r-s} \max \left( \| f \|_{C_0^s(\R^d)}, \| f \|_{H^s(\R^d)} \right) \label{eq:Sob_norm},
\end{equation}
where $C_{s,d,k_r}$ is a constant only depending on $r$, $s$, $d$, and $k_r$ (note that the dependency on the kernel $k_r$ is via the constant $C_1$).
\end{lemma}
\begin{proof}
As in Remark \ref{rem:append-assumptions}, we have $\sigma \geq 1$.
We then have
\begin{eqnarray*}
\| f_\sigma \|_{ \H_{k_r}  }^2 &=&  \int_{\| \xi \| \leq \sigma}  | \hat{f_\sigma}(\xi) |^2 \hat{\Phi}(\xi)^{-1} d\xi \quad (\because f \in \mathcal{B}_\sigma) \\ 
&\leq& C_1^{-1} \int_{\| \xi \| \leq \sigma}  | \hat{f_\sigma}(\xi) |^2 (1 + \| \xi \|^2)^r  d\xi \\
&\leq& C_1^{-1} (1 + \sigma^2)^{r-s}  \int_{\| \xi \| \leq \sigma}  | \hat{f_\sigma}(\xi) |^2  (1 + \| \xi \|^2)^s  d\xi \quad (\because r - s \geq 0) \\
&\leq&  C_1^{-1} (1 + \sigma^2)^{r-s} \| f _\sigma\|_{H^s(\R^d)}^2 
\leq C_1^{-1} 2^{r-s} \sigma^{2(r-s)}   \| f_\sigma \|_{H^s(\R^d)}^2 \quad (\because \sigma \geq 1).
\end{eqnarray*}
Therefore, by using Proposition \ref{prop:NarWar04_Prop3_11}, it follows that
\begin{eqnarray*}
\| f_\sigma \|_{ \H_{k_r}  } 
&\leq&  C_1^{-1/2} 2^{(r-s)/2} \sigma^{r-s}   \| f_\sigma \|_{H^s(\R^d)} \\
&\leq& C_1^{-1/2} 2^{(r-s)/2} \sigma^{r-s} C_{s,d}  \max \left( \| f \|_{C_0^s(\R^d)}, \| f \|_{H^s(\R^d)} \right),
\end{eqnarray*}
where $C_{s,d}$ is a constant only depending on $s$ and $d$.
The proof completes by setting $ C_{s,d,k_r} := C_1^{-1/2} 2^{(r-s)/2} C_{s,d}$.
\end{proof}

\section{Approximation in Sobolev spaces} 
\subsection{Fundamental lemma} \label{sec:appendix_lemma}
In the proof of Theorem \ref{theo:sob_weight}, we used Proposition 3.7 of \cite{NarWar04}, which assumes the existence of a function $\psi: \R^d \to \R$ satisfying the properties in Lemma \ref{lemma:existence}.
Since the existence of this function is not proved in \cite{NarWar04}, we will first prove it for completeness. 
Lemma \ref{lemma:existence} is a variant of Lemma 1.1 of \cite{FraJawWei91}, from which we borrowed the proof idea.

\begin{lemma} \label{lemma:existence}
Let $s\in\N$. Then there exists a function $\psi: \R^d \to \R$ satisfying the following properties:
\begin{itemize}
\item[(a)] $\psi$ is radial;
\item[(b)] $\psi$ is a Schwartz function;
\item[(c)] ${\rm supp}(\hat{\psi}) \subset B(0,1)$;
\item[(d)] $\int_{\R^d} x^\beta \psi (x) dx = 0$ for every multi-index $\beta$ satisfying $| \beta | := \sum_{i=1}^d \beta_i \leq s$, where $x^\beta := \prod_{i=1}^d x_i^{\beta_i}$.
\item[(e)] $\psi$ satisfies 
\begin{equation} \label{eq:Calderon_assump}
\int_0^\infty | \hat{\psi}(t \xi) |^2 \frac{dt}{t} = 1,\quad \forall \xi \in \R^d \backslash \{0\}.
\end{equation}
\end{itemize} 
\end{lemma}

\begin{proof}
Define a function $u \in L_1(\R^d)$ as the inverse Fourier transform of a function $\hat{u} \in L_1(\R^d)$ defined by $\hat{u}(\xi) := \exp \left(- 1 / (1- \| \xi \|^2) \right)$ if $\| \xi \| < 1$ and $\hat{u}(\xi) = 0$ otherwise.
Then $\hat{u}$ is radial, Schwartz, and satisfies ${\rm supp}(\hat{u}) \subset B(0,1)$ \citep[Sec.~2.28]{AdaFou03}.
Also note that $u$ is real-valued, since $\hat{u}$ is symmetric.

Let $m \in \mathbb{N}$ satisfy $m > s/2$.
Define a function $h: \R^d \to \R$ by 
$h := \Delta^m u$,
 where $\Delta$ denotes the Laplacian defined by $\Delta f := \sum_{i=1}^d \frac{\partial^2 f}{\partial x_i^2} $.
Note that we have (see e.g.~p.117 of \citealp{Ste70})
\begin{equation} \label{eq:laplacian_fourier_199}
\hat{h}(\xi) = C_m \| \xi \|^{2m} \hat{u}(\xi),
\end{equation}
where $C_m$ is a constant depending only on $m$. 
From this expression, it follows that $\hat{h}$ is radial and Schwartz (and so is $h$),  and that ${\rm supp}(\hat{h}) \subset B(0,1)$. 
Thus the function $h$ satisfies the required properties (a) (b) and (c).
Later we will define the function $\psi$ in the assertion based on $h$.

We next show that $h$ satisfies the property (d).
Let $\beta \in \N_0^d$ be any multi-index satisfying $| \beta | \leq s$, and let $p_\beta (x) := x^\beta$. 
It follows that $p_\beta h$ is Schwartz, and thus $p_\beta h \in L_1(\R^d)$.
Then  we have 
\begin{equation} \label{eq:integral_mono_211}
\int x^\beta h(x) dx = (\widehat{p_\beta h})(0),
\end{equation}
which follows from $p_\beta h \in L_1(\R^d)$ and from the definition of Fourier transform.
Note that we have $\widehat{p_\beta h}(\xi) = i^{| \beta |} \partial^{ \beta } \hat{h} (\xi)$, which can be expanded as
\begin{equation}
 \partial^{ \beta } \hat{h} (\xi)
\stackrel{\eqref{eq:laplacian_fourier_199}}{=}  \partial^\beta \left[ C_m \| \xi \|^{2m} \hat{u}(\xi) \right]
= C_m  \sum_{\gamma \in \N_0^d: \gamma \leq \beta} \binom{\beta}{\gamma} \partial^\gamma \left[\| \xi \|^{2m} \right] \partial^\beta \left[ \hat{u}(\xi) \right], \label{eq:mixed_partial_223}
\end{equation}
where $\gamma \leq \beta$ is defined by that $\gamma_i \leq \beta_i$ for all $i=1,\dots,d$, and $\binom{\beta}{\gamma} := \frac{ \prod_{i=1}^d \beta_i ! }{ \prod_{i=1}^d \gamma_i! }$. 
Using the multinomial theorem, the mixed partial derivative $\partial^\gamma \left[\| \xi \|^{2m} \right]$ in the above equation can be further expanded as
\begin{eqnarray} 
\partial^\gamma \left[\| \xi \|^{2m} \right] 
&=& \sum_{\alpha \in \N_0^d: | \alpha | = m} \frac{m!}{ \prod_{i=1}^d \alpha_i ! } \prod_{i=1}^d \frac{d^{ \gamma_i }}{d \xi_i^{\gamma_i}} \left[ \xi_i^{2 \alpha_i} \right] . \label{eq:multinomial_231}
\end{eqnarray}
From this it follows that
$\left. \partial^\gamma \left[\| \xi \|^{2m} \right] \right|_{\xi = 0} = 0$,
and thus (\ref{eq:mixed_partial_223}) gives that
$\partial^{ \beta } \hat{h} (0) = 0$.
Therefore, from (\ref{eq:integral_mono_211}) and $\widehat{p_\beta h}(\xi) = i^{| \beta |} \partial^{ \beta } \hat{h} (\xi)$, it holds that
$\int_\R~ x^\beta h(x) dx  = 0$,
which is the property (d). 

We next show that $\int_0^\infty | \hat{h}(t \xi) |^2 \frac{dt}{t} < \infty$ for all $\xi \in \R^d \backslash \{0\}$.
Since $\hat{h}$ is bounded and ${\rm supp} (\hat{h}) \subset B(0,1)$, we have $\int_1^\infty | \hat{h}(t \xi) |^2 \frac{dt}{t} < \infty$.
Also, since $| \hat{h}(t\xi) | = O(t^{2m})$ as $t \to +0$ (which follows from $\hat{h}(t\xi) = (-1)^m \| t\xi \|^{2m} \hat{u}(t \xi)$ with $\hat{u}$ being bounded), we have $\int_0^1  | \hat{h}(t \xi) |^2 \frac{dt}{t} < \infty$.
Therefore $\int_0^\infty | \hat{h}(t \xi) |^2 \frac{dt}{t} < \infty$.

Note that since $\hat{h}$ is radial, $\int_0^\infty | \hat{h}(t \xi) |^2 \frac{dt}{t}$ only depends on the norm $\| \xi \|$. 
Furthermore, $\int_0^\infty | \hat{h}(t \xi) |^2 \frac{dt}{t}$ remains the same for different values of the norm $\| \xi \| > 0$ due to the property of the Haar measure $dt/t$.
In other words, there is a constant $0 < C < \infty$ satisfying $\int_0^\infty | \hat{h}(t \xi) |^2 \frac{dt}{t} = C$ for all $\xi \in \R^d \backslash \{0\}$. 
The proof is completed by defining $\psi$ in the assertion as $\psi(x) := C^{-1/2} h(x)$.
\end{proof}

\noindent
\textbf{Notation.}
Note that $\psi$ being radial implies that $\hat{\psi}$ is radial, so $\hat{\psi}(t \xi)$ in (\ref{eq:Calderon_assump}) depends on $\xi$ only through its norm $\| \xi \|$.
Therefore we may henceforth use the notation 
\begin{equation} \label{eq:notation_psi_xi}
 \hat{\psi}(t \| \xi \|)
\end{equation}
to denote $\hat{\psi}(t \xi)$, to emphasize its dependence on the norm.
Similarly, we use the notation $\hat{\psi}(t)$ to imply $\hat{\psi}(t \xi)$ for some (and any) $\xi \in \R^d$ with $\| \xi \| = 1$.

\subsection{Approximation via Calder\'{o}n's  formula} \label{sec:sob_approx} 

The following result is known as Calder\'{o}n's  formula \citep[Theorem 1.2]{FraJawWei91}, and will be used in defining an approximate function \eqref{eq:approx_g}.
We use below the notation $f * g$ for any functions $f: \R^d \to \R$ and $g:  \R^d \to \R$ to denote their convolution: $(f * g) (x) := \int f(x-y)g(y) dy$.
\begin{theorem}[Calder\'{o}n's  formula]
Let $\psi \in L_1$ be a radial function satisfying (\ref{eq:Calderon_assump}), and for $t > 0$ define
\begin{equation} \label{eq:mollifier_286}
\psi_t(x) := \frac{1}{t^d}\psi(x/t), \quad x \in \R^d.
\end{equation}
Then for any $f \in L_2$, we have
\begin{equation} \label{eq:Calderon_formula}
f(x) = \int_0^\infty ({\psi_t} * \psi_t * f)(x)\ \frac{dt}{t}, \quad x \in \mathbb{R}^d,
\end{equation}
where the improper integral in (\ref{eq:Calderon_formula}) is to be interpreted in the following $L_2$ sense: if $0 < \varepsilon < \delta < \infty$ and $f_{\varepsilon, \delta}(x) := \int_{\varepsilon}^\delta ({\psi_t} * \psi_t * f)(x)\frac{dt}{t}$, then $\| f - f_{\varepsilon,\delta}\|_{L_2} \to 0$ as $\varepsilon \to +0$ and $\delta \to \infty$ independently.
\end{theorem}

Note that it is easy to verify from (\ref{eq:mollifier_286}) that
$\| \psi \|_{L_1} = \| \psi_t \|_{L_1}$ holds for all $t > 0$.
Let $\psi$ be the function in Lemma \ref{lemma:existence}. 
Following Section 3.2 of \cite{NarWar04}, we consider the following approximation of $f$ based on Calder\'{o}n's  formula (\ref{eq:Calderon_formula}):
\begin{equation} \label{eq:approx_g}
g_\sigma(x) := \int_{1/\sigma}^\infty ({\psi}_t * \psi_t * f)(x)\ \frac{dt}{t}.
\end{equation}
The integral in (\ref{eq:approx_g}) is also improper and should be interpreted as follows. 
Let $\delta > 1/\sigma$ and define 
\begin{equation} \label{eq:g_sigma_delta}
g_{\sigma, \delta} :=  \int_{1/\sigma}^\delta ({\psi}_t * \psi_t * f)(x)\ \frac{dt}{t}.
\end{equation}
Then $g_\sigma$ in (\ref{eq:approx_g}) is defined to be a function in $L_2$ such that $\lim_{\delta \to \infty} \| g_\sigma - g_{\sigma,\delta} \|_{L_2} = 0$.
Such $g_\sigma$ exists (as a limit of $g_{\sigma,\delta}$), as shown in Lemma \ref{lemma:g_sigma_fourier_416} below.
Since there is no proof of this result in \cite{NarWar04}, we provide a proof for the sake of completeness.
To this end, we first need the following lemma.

\begin{lemma} \label{lemma:g_sigma_L1L2}
Let $g_{\sigma,\delta}$ be defined as in \eqref{eq:g_sigma_delta} with $\delta > 1/\sigma$.
For all $1 \leq p \leq \infty$, if $f\in L_p$, then $g_{\sigma,\delta} \in L_p$. 
\end{lemma}
\begin{proof}
For $1 \leq p \leq \infty$, we have
\begin{eqnarray*}
\| g_{\sigma, \delta} \|_{L_p} 
&=&  \left\Vert   \int_{1/\sigma}^\delta {\psi}_t * \psi_t * f \frac{dt}{t} \right\Vert_{L_p} 
\leq  \int_{1/\sigma}^\delta  \left\Vert   {\psi}_t * \psi_t * f  \right\Vert_{L_p}   \frac{dt}{t} \quad (\because {\rm Minkowski\mathchar`'s\ inequality})  \\ 
&\le&   \int_{1/\sigma}^\delta \| \psi_t \|_{L_1}^2 \| f \|_{L_p} \frac{dt}{t}  \quad (\because {\rm Young\mathchar`'s\ inequality})  \\ 
&=&   \int_{1/\sigma}^\delta \| \psi \|_{L_1}^2 \| f \|_{L_p} \frac{dt}{t} = \| \psi \|_{L_1}^2 \| f \|_{L_p} (\log(\delta) - \log(1/\sigma)) < + \infty, \nonumber
\end{eqnarray*}
where in the last line we used the assumption $f \in L_p$ and the fact $\psi \in L_1$, which is a consequence of $\psi$ being a Schwartz function (see Lemma \ref{lemma:existence}). 
\end{proof}

\begin{lemma} \label{lemma:g_sigma_delta_fourier}
Assume $f \in L_1$, and let $g_{\sigma,\delta}$ be defined as in \eqref{eq:g_sigma_delta} with $\delta > 1/\sigma$.
Then the Fourier transform of $g_{\sigma,\delta}$ is given by
\[
\hat{g}_{\sigma,\delta}(\xi) =
\begin{cases} 
\hat{f}(\xi)  \int_{\| \xi \| / \sigma}^{\min(1, \| \xi \| \delta)} (\hat{\psi} (t) )^2  \frac{dt}{t}, \quad {\rm if}\ \| \xi \| < \sigma\\
0,\quad {\rm otherwise}
\end{cases}.
\]
\end{lemma}
\begin{proof}
We have
\begin{eqnarray*}
\hat{g}_{\sigma,\delta}(\xi) 
&=& \int \int_{1/\sigma}^\delta ({\psi}_t * \psi_t * f)(x) \frac{dt}{t} e^{- i \xi^T x} dx \\
&=& \int_{1/\sigma}^\delta  \int  ({\psi}_t * \psi_t * f)(x) e^{- i \xi^T x} dx  \frac{dt}{t} \quad (\because {\rm Fubini\mathchar`'s\ theorem}) \\
&=& \hat{f}(\xi)  \int_{1/\sigma}^\delta (\hat{\psi}_t (\xi) )^2  \frac{dt}{t}
= \hat{f}(\xi)  \int_{1/\sigma}^\delta (\hat{\psi} (t\xi) )^2  \frac{dt}{t}.
\end{eqnarray*}
In the above derivation, Fubini's theorem is applicable since $\psi_t * \psi_t * f \in L_1$ (which follows from $\psi \in L_1$, $f \in L_1$ and Young's inequality; see the proof of Lemma \ref{lemma:g_sigma_L1L2}).

Recall that $\hat{\psi}$ is radial, so that the value of $\hat{\psi}(t \xi)$ only depends on the norm of its argument $\| t \xi \| = t \| \xi \|$.
By a change of variables $\tau := t \| \xi \|$, and recalling the notation $\hat{\psi}(t \| \xi \| ) := \hat{\psi}(t\xi)$, it holds that
\begin{eqnarray}
\int_{1/\sigma}^\delta (\hat{\psi} (t \|\xi\| ) )^2  \frac{dt}{t} 
&=&\int_{ \| \xi \| / \sigma }^{ \| \xi \| \delta } ( \hat{\psi}(\tau) )^2 \frac{d\tau}{\tau} \nonumber\\
&=&
\begin{cases} 
\int_{\| \xi \| / \sigma}^{\min(1, \| \xi \| \delta)} (\hat{\psi} (\tau) )^2  \frac{d\tau}{\tau}, \quad {\rm if}\ \| \xi \| < \sigma\\
0,\quad {\rm otherwise}
\end{cases},
\end{eqnarray}
where the last line follows from the property ${\rm supp}(\psi) \subset B(0,1)$.
The proof is completed by combining this and the above expression of $\hat{g}_{\sigma,\delta}(\xi)$.
\end{proof}

We are now ready to show that the improper integral in (\ref{eq:approx_g}) is well-defined as a limit of $g_{\sigma,\delta}$ in $L_2$: The following lemma characterizes this limiting function in $L_2$ in terms of its Fourier transform.

\begin{lemma} \label{lemma:g_sigma_fourier_416}
Assume $f \in L_1 \cap L_2$.
Let $g_{\sigma,\delta}$ be defined as in \eqref{eq:g_sigma_delta} with $\delta > 1/\sigma$, and  $g_{\sigma} \in L_2$ be the inverse Fourier transform of $\hat{g}_\sigma \in L_2$ defined by
\[
\hat{g}_{\sigma}(\xi) =
\begin{cases}
\hat{f}(\xi)  \int_{\| \xi \| / \sigma}^1 (\hat{\psi} (t) )^2  \frac{dt}{t}, \quad {\rm if}\ \| \xi \| < \sigma\\
0, \quad {\rm otherwise}
\end{cases}.
\]
Then we have $\lim_{ \delta \to \infty } \| g_\sigma - g_{\sigma,\delta} \|_{L_2} = 0$.
\end{lemma}

\begin{proof}
First note that by Lemma \ref{lemma:g_sigma_L1L2}, the assumption $f \in L_1 \cap L_2$ implies $g_{\sigma,\delta} \in L_1 \cap L_2$, so we have $\hat{g}_{\sigma,\delta} \in L_1 \cap L_2$.
Below we will show $\lim_{ \delta \to \infty } \| \hat{g}_\sigma - \hat{g}_{\sigma,\delta} \|_{L_2} = 0$, from which the assertion follows because of the Fourier transform being an isometry from $L_2$ to $L_2$.
By Lemma \ref{lemma:g_sigma_delta_fourier} (which is applicable as $f \in L_1$) we have
\begin{eqnarray*}
\| \hat{g}_\sigma - \hat{g}_{\sigma,\delta} \|_{L_2}^2 
&=&  \int_{ \| \xi \| < \sigma } |\hat{f}(\xi) |^2 \left|  \int_{ \min(1, \| \xi \| \delta) }^{1} (\hat{\psi} (t) )^2  \frac{dt}{t}  \right|^2 d\xi.
\end{eqnarray*}
Therefore,
\begin{eqnarray}
\lim_{\delta \to \infty}  \| \hat{g}_\sigma - \hat{g}_{\sigma,\delta} \|_{L_2}^2
&=&  \int_{ \| \xi \| < \sigma } |\hat{f}(\xi) |^2  \lim_{\delta \to \infty}   \left|  \int_{ \min(1, \| \xi \| \delta) }^{1} (\hat{\psi} (t) )^2  \frac{dt}{t}  \right|^2 d\xi \label{eq:dominated_490} \\
&=&  \int_{ \| \xi \| < \sigma } |\hat{f}(\xi) |^2     \left|  \int_{1}^{1} (\hat{\psi} (t) )^2  \frac{dt}{t}  \right|^2 d\xi=0, \nonumber
\end{eqnarray}
where (\ref{eq:dominated_490}) follows from the dominated convergence theorem (which follows from $f \in L_2$).
\end{proof}

\subsection{The Sobolev norm of the approximate function}

In the main body of the paper, we use the following lemma, which is not provided in \cite{NarWar04}.

\begin{lemma} \label{eq:gsigma_norm}
Let $r,s\in\mathbb{R}$, $r, s > 0$ such that  $r \geq s$ and let $\sigma > 0$ be a constant. 
If $f \in H^s(\R^d) \cap L_1(\R^d)$, the function $g_\sigma$ defined in \eqref{eq:approx_g} satisfies
\begin{equation}
\| g_\sigma \|_{H^r} \leq (1 + \sigma^2)^{\frac{r-s}{2}} \| f \|_{H^s},\nonumber
\end{equation}
where $C > 0$ is a constant independent of $f$ and $\sigma$.
\end{lemma}
\begin{proof}

Note that from (\ref{eq:Calderon_assump}), if $\| \xi \| < \sigma$, we have 
$
\int_{\| \xi \| /\sigma}^1 | \hat{\psi}(t) |^2 \frac{dt}{t} \leq \int_{0}^1 | \hat{\psi}(t) |^2 \frac{dt}{t} \leq 1.
$
Therefore by Lemma \ref{lemma:g_sigma_fourier_416} we have 
\begin{eqnarray*}
\| g_\sigma \|_{H^r}^2 
&=&   \int_{B(0,\sigma)} (1 + \|\xi \|^2)^r | \hat{g_\sigma}(\xi) |^2 d\xi \\
&\leq&   \int_{B(0,\sigma)} (1 + \|\xi \|^2)^r | \hat{f}(\xi) |^2 d\xi \\
&\leq&   (1 + \sigma^2)^{r-s} \int_{B(0,\sigma)}  (1 + \|\xi \|^2)^s | \hat{f}(\xi) |^2 d\xi \\
&\leq&   (1 + \sigma^2)^{r-s}  \| f \|_{H^s}^2,
\end{eqnarray*}
yielding the result.
\end{proof}

\bibliographystyle{spmpsci}      

\bibliography{main.bbl}

%
%

\end{document}